\documentclass[11pt, reqno]{amsart}
\usepackage{amsmath, amsthm, amsfonts, amssymb, amsbsy, bigstrut, graphicx, enumerate,  upref, longtable, comment, booktabs, array, caption, subcaption}

\usepackage{pstricks}
\usepackage{times}

\usepackage[numbers, sort&compress]{natbib}

\usepackage[breaklinks]{hyperref}

\newcommand{\vol}{\mathrm{Vol}}
\newcommand{\I}{\mathrm{i}}

\newcommand{\ms}{\mathcal{S}}

\newcommand{\tf}{\tilde{f}}

\newtheorem{thm}{Theorem}[section]
\newtheorem{lmm}[thm]{Lemma}
\newtheorem{cor}[thm]{Corollary}
\newtheorem{prop}[thm]{Proposition}
\newtheorem{defn}[thm]{Definition}

\theoremstyle{definition}

\newcommand{\ee}{\mathbb{E}}

\newcommand{\hess}{\operatorname{Hess}}

\newcommand{\cp}{\mathcal{P}}

\newcommand{\pp}{\mathbb{P}}

\newcommand{\rr}{\mathbb{R}}
\newcommand{\smallavg}[1]{\langle #1 \rangle}

\newcommand{\ve}{\varepsilon}

\newcommand{\zz}{\mathbb{Z}}

\newcommand{\fpar}[2]{\frac{\partial #1}{\partial #2}}

\numberwithin{equation}{section}

\newcommand{\eq}[1]{\begin{align*} #1 \end{align*}}

\usepackage[utf8]{inputenc}
\usepackage{amsmath}
\usepackage{amsfonts}
\usepackage{bbm}
\usepackage{mathtools}
\usepackage{tikz}
\usetikzlibrary{decorations.pathreplacing}
\usepackage{amsthm}
\usepackage[english]{babel}
\usepackage{fancyhdr}
\usepackage{amssymb}
\usepackage{enumitem}
\usepackage{qtree}
\usepackage{mathrsfs}

\renewcommand{\bar}{\overline}

\renewcommand{\tilde}{\widetilde}

\begin{document}
\title{Local KPZ behavior under arbitrary scaling limits}
\author{Sourav Chatterjee}

\address{Departments of mathematics and statistics, Stanford University}
\email{souravc@stanford.edu}
\thanks{Research partially supported by NSF grants DMS-1855484 and DMS-2113242}
\thanks{Data availability statement: Data sharing not applicable to this article as no datasets were generated or analysed during the current study.}
\thanks{Conflict of interest statement: The author has no competing interests to declare that are relevant to the content of this article.}

\begin{abstract}
One of the main difficulties in proving convergence of discrete models of surface growth to the Kardar--Parisi--Zhang (KPZ) equation in dimensions higher than one is that the correct way to take a scaling limit, so that the limit is nontrivial, is not known in a rigorous sense. To understand KPZ growth without being hindered by this issue, this article introduces a notion of  `local KPZ behavior', which roughly means that the instantaneous growth of the surface at a point decomposes into the sum of a Laplacian term, a gradient squared term, a noise term that behaves like white noise, and a remainder term that is negligible compared to the other three terms and their sum. The main result is that for a general class of surfaces, which contains the model of directed polymers in a random environment as a special case, local KPZ behavior occurs under \emph{arbitrary scaling limits}, in any dimension. 
\end{abstract}

\maketitle

\section{Introduction}
\subsection{The KPZ equation}
The Kardar--Parisi--Zhang (KPZ) equation was introduced in~\cite{kardaretal86} to model the growth of a generic randomly growing surface. If $f(t,x)$ is the height of a $d$-dimensional surface at time $t\in \rr_{\ge0}$ and location $x\in \rr^d$, the KPZ equation prescribes that the evolution of $f$ is governed by the equation 
\begin{align}\label{kpzeq}
\partial_t f = \nu \Delta f + \frac{\lambda}{2}  |\nabla f|^2 + \sqrt{D} \xi,
\end{align}
where $\xi$ is a random field known as space-time white noise, and $\nu$, $\lambda$ and $D$ are the   parameters of the model. Formally, space-time white noise is a distribution-valued centered Gaussian random field, with covariance structure
\[
\ee(\xi(t,x)\xi(t',x')) = \delta(t-t') \delta^{(d)}(x-x'),
\]
where $\delta$ and $\delta^{(d)}$ are the Dirac delta functions on $\rr$ and $\rr^d$, respectively. (See Section \ref{whitesec} for a precise definition of space-time white noise.)

It is difficult to give a rigorous meaning to the KPZ equation, mainly due to the well-known difficulties in defining products of distributions. This problem now has a complete solution in dimension one, using a variety of techniques,  such as the Cole--Hopf solution~\cite{bertinigiacomin97}, regularity structures~\cite{hairer13, hairer14}, paracontrolled distributions~\cite{gubinellietal15, gubinelliperkowski17}, energy solutions~\cite{goncalvesjara12, goncalvesjara14, goncalvesetal15, gubinelliperkowski18}, and renormalization group~\cite{kupiainenmarcozzi17}. Moreover, many one-dimensional discrete processes have been shown to have a KPZ scaling limit, as in~\cite{albertsetal14b, amiretal11, borodincorwin14, borodinetal13, dembotsai16, dotsenko10, prahoferspohn02,  sasamotospohn10, yang20}.  All of this is only a small sample of the enormous literature that has grown around rigorous 1D KPZ. For surveys, see~\cite{corwin16, quastel12, quastelspohn15}. 

There are some recent constructions of distribution-valued solutions of the KPZ equation in dimensions greater than one~\cite{caravennaetal20, chatterjeedunlap20, cometsetal19, cometsetal20,   dunlapetal20,   gu20,  lygkoniszygouras22,magnenunterberger18, coscoetal20}. These solutions are `physically trivial', by being equivalent to solutions of a linear stochastic differential equation, called the stochastic heat equation with additive noise. A `nontrivial' solution of the KPZ equation in $d\ge 2$ has not  yet been constructed, although a promising breakthrough has occurred very recently for the related 2D stochastic heat equation with multiplicative noise, which is formally the `exponential' of 2D KPZ~\cite{caravennaetal21}. A more detailed discussion of all this is in the forthcoming sections. 

A fundamental roadblock in constructing nontrivial solutions of the KPZ equation in $d\ge 2$ is that we do not know {\it how to take scaling limits} of approximate solutions to reach a nontrivial limit. Even in dimension one, there can be many different scaling limits. See, for example, \cite[Section 7]{albertsetal14b} for a discussion of the various ways of taking scaling limits of 1D directed polymers, only one of which has been made fully rigorous. But in many 1D  models, we know at least one way of taking a scaling limit that leads to a nontrivial solution of the KPZ equation. In higher dimensions, the question becomes less tractable. Physicists believe that for 2D models, the celebrated `Family--Vicsek scaling'~\cite{familyvicsek91} is the correct one, and leads to a function-valued, rather than distribution-valued, solution of the 2D KPZ equation. This has been verified in numerical simulations~\cite{kellingodor11, rodriguesetal14, halpin-healy12} for discrete models, but remains out of the reach of rigorous mathematics. (See the end of Subsection \ref{results} for a more detailed discussion.) 


\subsection{Local KPZ growth}\label{localsec}
The goal of this paper is to take a small step towards understanding KPZ in $d\ge 2$ without running into the issue of constructing scaling limits, building on a framework introduced recently in the series of papers~\cite{chatterjee21, chatterjee21b, chatterjeesouganidis21}. (Even in $d=1$, this new framework may be useful in going beyond exactly solvable models; this will appear in forthcoming work with Arka Adhikari.) Since the `correct' way to scale is still mysterious, the following workaround is proposed. Consider a general class of growth models, which contains at least one model of widespread interest. Then show that, {\it irrespective of how we take a scaling limit}, the growth is always {\it locally} like the KPZ equation \eqref{kpzeq}, breaking up as the sum of a Laplacian term, a gradient squared term, a noise term, and a residual term that is negligible compared to the other three terms and their sum. Surprisingly, this turns out to be doable. The details are as follows. 


The first step is to give a precise definition of local KPZ growth. Take any $d\ge 1$. Suppose that we have a collection of random functions $\{f_\ve\}_{\ve >0}$ from $\zz_{\ge 0}\times \zz^d$ into $\rr$. A general `rescaling' of $f_\ve$ is defined as follows.  Let $\alpha(\ve)$, $\beta(\ve)$ and $\gamma(\ve)$ be positive real numbers depending on $\ve$, with $\alpha(\ve)$ and $\beta(\ve)$ tending to zero as $\ve \to 0$. Based on  these coefficients, the rescaled version of $f_\ve$ is the function $f^{(\ve)} : \rr_{>0} \times \rr^d \to \rr$ defined as 
\[
f^{(\ve)}(t,x) := \gamma(\ve) f_\ve(\lceil\alpha(\ve)^{-1} t\rceil, \lceil\beta(\ve)^{-1} x\rceil),
\]
where $\lceil u\rceil$ denotes the smallest integer greater than or equal to $u$ when $u\in \rr$, and denotes the vector $(\lceil u_1\rceil,\ldots, \lceil u_d\rceil )$ when $u=(u_1,\ldots,u_d)\in \rr^d$. Note that this means space and time are rescaled so that successive time points are separated by $\alpha(\ve)$ and neighboring points in space are separated by $\beta(\ve)$. The factor $\gamma(\ve)$ is just a multiplicative factor meant to ensure that the limit of $f^{(\ve)}$ as $\ve \to 0$ (on some appropriate space of functions or distributions) does not blow up to infinity or shrink to zero. This is why we need $\alpha(\ve)$ and $\beta(\ve)$ to tend to zero, but there is no restriction on $\gamma(\ve)$. 

Let $A = \{0,\pm e_1,\ldots, \pm e_d\}$ be the set consisting of the origin and its nearest neighbors in $\zz^d$. Define the `local average' of $f^{(\ve)}$ at a point $(t,x)\in \rr_{>0}\times \rr^d$ as
\[
\bar{f}^{(\ve)}(t,x) := \frac{1}{2d+1}\sum_{a\in A} f^{(\ve)}(t, x+\beta(\ve) a), 
\]
the `approximate time derivative' of $f^{(\ve)}$ as
\[
\tilde{\partial}_t f^{(\ve)}(t,x) := \frac{f^{(\ve)}(t+\alpha(\ve), x) - f^{(\ve)}(t,x)}{\alpha(\ve)},
\]
the `approximate Laplacian' as 
\[
\tilde{\Delta} f^{(\ve)}(t,x) := (2d+1)\biggl(\frac{\bar{f}^{(\ve)}(t,x) - f^{(\ve)}(t,x)}{\beta(\ve)^2}\biggr)
\]
and the `approximate squared gradient'\footnote{In the definition of the approximate squared gradient, one may object that it is more natural to have $f^{(\ve)}(t,x)$ instead of $\bar{f}^{(\ve)}(t,x)$ as the term to be subtracted off. The reason behind choosing $\bar{f}^{(\ve)}(t,x)$ is that if we use $f^{(\ve)}(t,x)$ instead, then we will end up with an extra $(\Delta f)^2$ term in the limiting equation, which is not present in the KPZ equation. The situation has some similarity with the definition of stochastic integral, where slightly different definitions give rise to two completely different equations (It\^{o} and Stratonovich).} as 
\[
|\tilde{\nabla} f^{(\ve)}(t,x)|^2 := \frac{1}{2}\sum_{a\in A} \biggl(\frac{f^{(\ve)}(t, x+\beta(\ve) a) - \bar{f}^{(\ve)}(t,x)}{\beta(\ve)}\biggr)^2.
\]
The above definitions are inspired by the fact that if $\alpha(\ve)\to0$, $\beta(\ve)\to 0$, and $f^{(\ve)}$ converges in some strong sense to a smooth function $f$ as $\ve \to 0$, then the approximate time derivative, the approximate Laplacian, and the approximate squared gradient converge to $\partial_t f$, $\Delta f$ and $|\nabla f|^2$. Of course, we do not expect $f^{(\ve)}$ to converge to a smooth limit in general.

For the definition of local KPZ behavior below, and for use in the rest of the paper, recall the meanings of the $o_P$ and $O_P$ notations. If $\{X_\ve\}_{\ve >0}$ and $\{Y_\ve\}_{\ve >0}$ are collections of random variables (which are allowed to be constants), we say that $X_\ve=o_P(Y_\ve)$ if for any $\delta >0$,
\[
\lim_{\ve \to 0} \pp(|X_\ve | > \delta |Y_\ve|) = 0.
\]
In other words, $X_\ve/Y_\ve \to 0$ in probability as $\ve \to 0$. Similarly, we say that $X_\ve = O_P(Y_\ve)$ if 
\[
\lim_{K\to \infty} \limsup_{\ve \to 0} \pp(|X_\ve| > K|Y_\ve|) = 0.
\] 
In other words, $\{X_\ve/Y_\ve\}_{\ve >0}$ is a tight family of random variables.
\begin{defn}[Local KPZ behavior]\label{kpzdef}
Let all notation be as above. We will say that $f^{(\ve)}$ has `local KPZ behavior' as $\ve \to 0$ if for some strictly positive $\nu(\ve)$, $\lambda(\ve)$, and $D(\ve)$, which can vary arbitrarily with $\ve$, some collection of `noise fields' $\xi^{(\ve)}: \rr_{>0} \times \rr^d \to \rr$,  and some collection of `remainder fields' $R^{(\ve)}: \rr_{>0} \times \rr^d \to \rr$,  we have that for any $(t,x)\in \rr_{>0}\times \rr^d$, 
\begin{align*}
\tilde{\partial}_t f^{(\ve)}(t,x) &= \nu(\ve) \tilde{\Delta} f^{(\ve)}(t,x) +\frac{ \lambda(\ve)}{2} |\tilde{\nabla} f^{(\ve)}(t,x)|^2 \\
&\qquad + \sqrt{D(\ve)}\xi^{(\ve)}(t,x) + R^{(\ve)}(t,x), 
\end{align*}
such that the following conditions hold: 
\begin{enumerate}
\item The noise field $\xi^{(\ve)}$ converges in law to white noise on $\rr_{> 0}\times \rr^d$ as $\ve \to 0$ (see Section \ref{whitesec} for the definition of this convergence).
\item The remainder term $R^{(\ve)}(t,x)$ is $o_P$ of the first three terms on the right and their sum, meaning that  $R^{(\ve)}(t,x) $ divided by any of the first three terms, or by their sum, tends to zero in probability as $\ve \to 0$.
\end{enumerate}
\end{defn}
Just to fully clarify the second condition and remove any scope for confusion, we note that it means that for any fixed $(t,x)$, the quantities
\[
\frac{R^{(\ve)}(t,x)}{\nu(\ve) \tilde{\Delta} f^{(\ve)}(t,x)}, \ \ \frac{R^{(\ve)}(t,x)}{\lambda(\ve)|\tilde{\nabla} f^{(\ve)}(t,x)|^2}, \ \ \frac{R^{(\ve)}(t,x)}{\sqrt{D(\ve)}\xi^{(\ve)}(t,x)}
\]
and 
\[
\frac{R^{(\ve)}(t,x)}{\nu(\ve) \tilde{\Delta} f^{(\ve)}(t,x) +\frac{ 1}{2}\lambda(\ve) |\tilde{\nabla} f^{(\ve)}(t,x)|^2  + \sqrt{D(\ve)}\xi^{(\ve)}(t,x)}
\]
all tend to zero in probability as $\ve \to 0$. 

In our examples, we will have that for fixed $(t,x)$ and $\ve$, the noise term $\xi^{(\ve)}(t,x)$ is independent of the Laplacian and gradient squared terms. But we omit this from the definition of local KPZ behavior, so as to leave open the possibility of other examples where the independence criterion does not hold.

It may seem as if $\nu(\ve)$, $\lambda(\ve)$, and $D(\ve)$ should not be allowed to vary with $\ve$ if we want something analogous to \eqref{kpzeq}. However, this is not true. In the KPZ literature, it is understood that the coefficients in \eqref{kpzeq} can be allowed to vary when taking a scaling limit, and even be allowed to tend to zero or blow up to infinity. This is especially true in dimensions higher than one. For example, the Family--Vicsek scaling for 2D surfaces~\cite{familyvicsek91} requires this (see further discussion in Subsection \ref{results}). The important point is that we want the time derivative to decompose into a linear combination of the Laplacian, the squared gradient, a noise term that behaves like white noise, and a negligible error term. This is captured by our definition of local KPZ growth. 

Having defined the notion of local KPZ growth, we define in the next subsection a class of discrete growth models that will be shown to have local KPZ growth under arbitrary scaling limits. 

\subsection{A class of growing random surfaces}\label{modelsec}
Fix some $d\ge 1$. Recall that we defined $A = \{0,\pm e_1,\ldots, \pm e_d\}$ to be the set consisting of the origin and its nearest neighbors in $\zz^d$. Let $\phi:\rr^A \to \rr$ be a function. Let $\mathbf{z} = \{z_{t,x}\}_{t\in \zz_{>0}, x\in \zz^d}$ be a collection of i.i.d.~random variables, which will be called the `discrete noise field', or simply the `noise field' when there is no scope for confusion with the noise field $\xi^{(\ve)}$ from Definition \ref{kpzdef}. Given $\ve >0$, consider a function $f_\ve: \zz_{\ge 0}\times \zz^d\to \rr$ growing as follows: $f_\ve(0,x)=0$ for all $x$, and for each $t\ge 0$,
\begin{align}\label{growthmodel}
f_\ve(t+1,x) := \phi((f_\ve(t,x+a))_{a\in A}) + \ve z_{t+1,x}. 
\end{align}
Imagine $f_\ve(t,x)$ to be the height of a $d$-dimensional random surface at time $t$ and location $x$. The above recursion says that the height at time $t+1$ is a function of the heights at $x$ and its neighbors at time $t$, plus an independent random fluctuation. Since the function $\phi$ `drives' the growth of $f$, we will sometimes refer to $\phi$ as the `driving function' (as in~\cite{chatterjee21, chatterjee21b}).

Let $1\in \rr^A$ denote the vector of all $1$'s. For $u\in \rr^A$, let $\bar{u}$ denote the average of the coordinates of $u$. For $u,v\in \rr^A$, let us write $u\ge v$ if $u_a \ge v_a$ for each $a\in A$.  We make the following assumptions about $\phi$. 

\begin{itemize}
\item {\it Equivariance under constant shifts.} We assume that for all $u\in \rr^A$ and $c\in \rr$,  $\phi(u + c1) = \phi(u)+c$. Besides being physically natural, this assumption has a long history in the literature on convergence of approximation schemes for partial differential equations, starting with \cite{barlessouganidis91}. It is also part of the framework introduced in~\cite{chatterjee21, chatterjee21b, chatterjeesouganidis21}.
\item {\it Zero at the origin.} We assume that $\phi(0)=0$. There is no loss of generality in this assumption, since equivariance ensures that if $\phi(0)\ne 0$, and we define $\tilde{\phi}(u) := \phi(u)-\phi(0)$, and $\tf_\ve$ is defined using $\tilde{\phi}$ in \eqref{growthmodel}, then $\tf_\ve(t,x) = f_\ve(t,x)-t\phi(0)$ for all $t$ and $x$.
\item {\it Monotonicity.} We assume that $\phi(u)\ge \phi(v)$ whenever $u\ge v$. This assumption, too, is physically natural and has appeared in related prior work~\cite{barlessouganidis91, chatterjee21, chatterjee21b, chatterjeesouganidis21}.
\item {\it Symmetry.} We assume that $\phi(u)$ remains unchanged under any permutation\footnote{One may say that it's more natural to require that $\phi$ is only symmetric in $\{u_a\}_{a\ne 0}$. Indeed, that is true, but it leads to messier notation and more cumbersome statements and proofs without adding to the intellectual content of the results, which is why we will work with the stronger assumption of complete symmetry. It should not be difficult to prove analogous results under a weaker symmetry assumption, or even under no symmetry at all, as was done in \cite{chatterjeesouganidis21} for the deterministic analogue of the setup considered here.} of the coordinates of $u$. This is a strengthening of the assumption of `invariance under lattice symmetries' from \cite{chatterjee21}.
\item {\it Regularity.} We assume that $\phi$ is differentiable everywhere, and twice continuously differentiable in a neighborhood of the origin. As noted in \cite{chatterjee21, chatterjeesouganidis21}, this assumption is needed for convergence to KPZ. In the absence of this assumption, the local growth may resemble some other equation, as in~\cite{chatterjeesouganidis21}. 
\item {\it Nondegeneracy.} We assume that the Hessian matrix of $\phi$ at the origin is nonzero. This assumption is needed to ensure the presence of the gradient squared term in the KPZ limit. If the Hessian at the origin is zero, we may have a different kind of local growth, as in~\cite{chatterjeesouganidis21}. 
\item {\it Strict Edwards--Wilkinson domination.} The Edwards--Wilkinson surface growth model~\cite{edwardswilkinson82} is described by equation \eqref{growthmodel} with $\phi(u)=\bar{u}$. We assume that our surface grows at least as fast as the Edwards--Wilkinson surface, meaning that $\phi(u)\ge \bar{u}$ for all $u$. Moreover, we assume that this domination is strict, in the following sense: If $\{u_n\}_{n\ge 1}$ is a sequence such that $\phi(u_n) - \bar{u}_n \to 0$, then $u_n - \bar{u}_n 1 \to 0$. This is one of the two key assumptions that allow us to deduce local KPZ behavior under arbitrary scaling limits.\footnote{Note that the Edwards--Wilkinson model itself does not satisfy the nondegeneracy condition stated above, nor the condition of strict Edwards--Wilkinson domination. Therefore, it does not fit into our framework. But that is all right, since the KPZ equation is not the scaling limit of this model.}
\end{itemize}
In addition to the above assumption on $\phi$, we also make the following set of assumptions on the noise field (in addition to the fact that it is a field of i.i.d.~random variables).
\begin{itemize}
\item {\it Zero mean.} We assume that the noise variables have zero mean. 
\item {\it Boundedness.} We assume that the noise variables are bounded. That is, there is some constant $B$ such that $|z_{t,x}|\le B$ almost surely. This is the second key assumption that ensures local KPZ growth under arbitrary scaling limits.
\item {\it Absolute continuity.} We assume that the law of the noise variables is absolutely continuous with respect to Lebesgue measure. We will refer to this condition by simply saying that the noise variables are `continuous'.
\end{itemize}
Under the above conditions on the driving function and the noise field, it turns out that the discrete surface $f_\ve$ has local KPZ growth under arbitrary scaling limits. This result is stated in the next subsection. 

\subsection{Results}\label{results}
Let $f_\ve$ be as in the previous subsection, and suppose that all of the stated assumptions on $\phi$ and the noise field are satisfied. Let $\alpha(\ve)$, $\beta(\ve)$ and $\gamma(\ve)$ be positive real numbers depending on $\ve$, such that $\alpha(\ve)$ and $\beta(\ve)$ tend to zero as $\ve \to 0$. As in Subsection \ref{localsec}, define the rescaled function $f^{(\ve)} : \rr_{>0} \times \rr^d \to \rr$ as 
\[
f^{(\ve)}(t,x) := \gamma(\ve) f_\ve(\lceil\alpha(\ve)^{-1} t\rceil, \lceil\beta(\ve)^{-1} x\rceil).
\]
The following theorem shows that $f^{(\ve)}$ has local KPZ behavior under any scaling where $\ve$ is sent to zero as the lattice spacing goes to zero. This is the main result of this paper.
\begin{thm}\label{kpzprob}
Under the assumptions on the driving function $\phi$ and the noise field $\mathbf{z}$ stated in the previous subsection, $f^{(\ve)}$ has local KPZ behavior as $\ve \to 0$, in the sense of Definition \ref{kpzdef}, for any choice of $\alpha(\ve)$, $\beta(\ve)$ and $\gamma(\ve)$ such that $\alpha(\ve)$ and $\beta(\ve)$ tend to zero as $\ve \to 0$. Moreover, the coefficients $\nu(\ve)$, $\lambda(\ve)$ and $D(\ve)$ of Definition \ref{kpzdef} turn out to be the following:
\begin{align*}
\nu(\ve) &= \frac{\beta(\ve)^2}{(2d+1)\alpha(\ve)}, \  \  \lambda(\ve) = \frac{2(q-r)\beta(\ve)^2}{\alpha(\ve)\gamma(\ve)},  \   \ D(\ve) =\frac{\sigma^2 \ve^2\beta(\ve)^{d}\gamma(\ve)^2}{\alpha(\ve)}, 
\end{align*}
where $\sigma^2$ is the variance of the noise variables, $q$ is the value of the diagonal elements of $\hess\phi(0)$ (which are all equal due to the symmetry of $\phi$), and $r$ is the value of the off-diagonal elements of $\hess \phi(0)$.
\end{thm}
To appreciate the meaning of Theorem \ref{kpzprob}, consider the following. It is not hard to guess that local KPZ behavior must be a consequence of Taylor expansion. But to invoke Taylor expansion, we need that for two neighboring points $x$ and $y$, $f_\ve(t,x)\approx f_\ve(t,y)$.  This is trivially true for $t=0$, since $f_\ve(0,x)=0$ for all $x$. Using~\eqref{growthmodel}, one can then deduce by a crude inductive argument that this continues to hold as long as $t$ does not exceed a threshold determined by $\ve$. But for Theorem~\ref{kpzprob} to be true, we need that $f_\ve(t,x)$ and $f_\ve(t,y)$ continue to be close to each other for neighboring $x$ and $y$ {\it even if $t$ is allowed to vary arbitrarily as $\ve \to 0$.} This is encapsulated by the following result, which is the key step in proving Theorem \ref{kpzprob}. Recall the $O_P$ notation defined above Definition \ref{kpzdef}.
\begin{thm}\label{oepthm}
Let $t_\ve \in \zz_{\ge 1}$ and $x_\ve\in \zz^d$ vary arbitrarily with $\ve$. Then, as $\ve \to 0$, $f_\ve(t_\ve, x_\ve) - f_\ve(t_\ve,x_\ve + a) = O_P(\sqrt{\ve})$ for each $a\in A$. 
\end{thm}

We will see later that the proof of Theorem \ref{oepthm} does not provide much intuition for why the result is true. An intuitive  explanation is indicated by the next result. For a function $f:\zz^d \to \rr$, let $\delta_i f(x) := f(x+e_i)-f(x)$ for $i=1,\ldots,d$. Let $\delta f := (\delta_1 f,\ldots, \delta_d f)$ be the `gradient' of $f$. Heuristically, it is possible that for fixed $\ve$, the `gradient field' $\delta f_\ve(t,\cdot)$ converges in law to a stationary process as $t\to \infty$. It is not clear whether this is true, but the following result gives strong evidence in favor, under one extra assumption.
\begin{thm}\label{tightthm}
Suppose that in addition to the hypotheses of Theorem \ref{kpzprob}, the following holds: Whenever $|u_n - \bar{u}_n1|\to \infty$, we have $\phi(u_n)-\bar{u}_n \to \infty$. Then for any fixed $\ve >0$, the sequence of gradient fields $\{\delta f_\ve(t,\cdot)\}_{t\in \zz_{\ge0}}$ is a time-homogeneous Markov chain on the state space $(\rr^d)^{\zz^d}$ (endowed with the product topology, which makes it a Polish space), with at least one translation-invariant stationary probability distribution. Moreover, it is a tight family. 
\end{thm}
In this context, it is worth noting that there has been an enormous amount of effort in the last two years on understanding stationary KPZ growth. Stationary solutions of the stochastic Burgers equation --- which is supposed to be the equation for the gradient field of the solution of the KPZ equation --- have been constructed in dimension one~\cite{bakhtinli19, dunlapetal21} and also in dimensions two and three~\cite{dunlap20}. Stationary solutions of the 1D `open KPZ' equation were recently constructed in~\cite{corwinknizel21} and further developed and analyzed in~\cite{bryckuznetsov21, brycetal21, barraquandledoussal21}. Convergence to stationarity has been studied in~\cite{pimentel18, pimentel21a, pimentel21b, yang21}. It would be interesting if similar results can be proved in the general setting of Theorem \ref{tightthm}.

Incidentally, the reason why the existence results in \cite{bakhtinli19, dunlapetal21, dunlap20} are restricted to $d\le 3$, while Theorem \ref{tightthm} holds in any dimension, is that the discreteness of both space and time in Theorem \ref{tightthm} makes it easy to overcome the problems of ill-posedness inherent in continuous-time differential equations. Although the white noise is smoothed in space in~ \cite{bakhtinli19, dunlapetal21, dunlap20}, which makes the space effectively discrete, it remains unsmoothed in time, giving rise to genuine technical limitations. This is the same reason why $\ve$ can be arbitrary in Theorem \ref{tightthm}, while it needs to be small enough for the results of  \cite{bakhtinli19, dunlapetal21, dunlap20}.

Another large number of recent results related to the above theorems are about local behaviors of solutions of the 1D KPZ equation and related processes, such as the Airy sheet, the KPZ line ensemble, the Brownian landscape, and the KPZ fixed point~\cite{basuetal21, dauvergneetal20, corwinhammond16, dasghosal21, sarkarvirag21, pimentel18, pimentel21a, pimentel21b}. These results contain much more information than Theorem \ref{kpzprob}, but for one-dimensional processes. Again, it would be interesting if some analogous refined results can be proved in the general setting considered above. 

Theorem \ref{kpzprob} can be viewed as a KPZ universality result. Roughly speaking, KPZ universality is the notion that the KPZ equation arises as the scaling limit of a large and varied class of growing random surfaces for which exact formulas are not available. Significant progress on 1D KPZ universality has been made in recent years~\cite{hairerquastel18, gubinelliperkowski18b, hairerxu19, yang20b, dembotsai16, yang21, yang20c}, although much remains to be understood. In dimensions higher than one, almost nothing is known. Theorem~\ref{kpzprob} is a small step towards understanding the universal nature of KPZ growth in general dimensions in the absence of integrability. 


As a final remark, suppose that $d=1$, and we want the coefficients $\nu(\ve)$, $\lambda(\ve)$ and $D(\ve)$ to {\it not} depend on $\ve$. Then by the formulas from Theorem~\ref{kpzprob}, we need that $\beta(\ve)^2 \propto \alpha(\ve)$, $\beta(\ve)^2\propto \alpha(\ve)\gamma(\ve)$, and $\alpha(\ve) \propto \ve^2 \beta(\ve)\gamma(\ve)$, where the proportionality constants may depend on $d$ and the law of the noise variables. The first two conditions show that $\gamma(\ve)$ must be a constant, and then plugging this into the third condition and using the first condition again, we get $\alpha(\ve)\propto \ve^4$. Then using the first condition one final time, we have $\beta(\ve)\propto \ve^2$. So, when $d=1$, the only way to ensure that $\nu(\ve)$, $\lambda(\ve)$ and $D(\ve)$ do not vary with $\ve$ is to have $\alpha(\ve)\propto \ve^4$, $\beta(\ve)\propto\ve^2$ and $\gamma(\ve)=$ constant. We will see later that for directed polymers in random environment, this gives the `intermediate disorder' scaling limit constructed in~\cite{albertsetal14b}. In forthcoming work with Arka Adhikari, it will be shown that a class of 1D surfaces (of the type considered in this paper) converge in law to this universal scaling limit (known as the Cole--Hopf solution of the 1D KPZ equation) under the above scaling of space and time.

Intriguingly, for $d\ge 2$, the same logic shows that there is no way to get constant coefficients as $\ve \to 0$, if we insist on $\alpha \to 0$ and $\beta \to 0$ as $\ve \to 0$. This suggests that at least one of the coefficients $\nu$, $\lambda$ and $D$ must tend to zero as $\ve \to 0$ for a KPZ scaling limit in $d\ge 2$. Indeed, numerical simulations (such as in~\cite{kellingodor11, rodriguesetal14, halpin-healy12}) suggest that for $d=2$, it may be possible to obtain a function-valued scaling limit by taking (in what is known in physics as the Family–Vicsek scaling~\cite{familyvicsek91}) $\nu \sim \beta^{2-z}$, $\lambda \sim \beta^{2-z-a}$, and $D \sim \beta^{2+2a-z}$ for certain exponents $a$ and $z$. Scaling arguments based on Galilean invariance~\cite{barabasistanley95} suggest that these exponents should satisfy $a + z = 2$. If we assume this, then we obtain the scaling $\nu \sim \beta^a$, $\lambda \sim 1$, and $D \sim  \beta^{3a}$. Theorem \ref{kpzprob} has something interesting to say here: Let $d=2$. Suppose we take some $\beta = \beta(\ve)\to 0$ as $\ve \to 0$, and let $\alpha= \beta^z$ for some exponent $z$. Suppose that in this setting, we want to have $\lambda \sim 1$ in Theorem \ref{kpzprob}. Then the formula for $\lambda$ implies that $\gamma\sim \beta^{2-z}$. Plugging this into the formulas for $\nu$ and $D$, we get $\nu \sim \beta^{2-z}$ and $D\sim \ve^2 \beta^{3(2-z)}$, exactly matching the Family--Vicsek scaling except for the $\ve^2$. However, this may not be an issue, since $\ve$ is often taken to scale like $|\log \beta|^{-1/2}$ in 2D (as in \cite{chatterjeedunlap20, caravennaetal17, caravennaetal20, caravennaetal21}), and therefore has no role to play in the exponents. Thus, it is possible that Theorem \ref{kpzprob} may provide a launchpad to an eventual rigorous proof of the Family--Vicsek scaling relation.

In this context, it should also be noted that the numerical works are exclusively for discrete models. There does not seem to be any numerical work for the continuum KPZ equation, although there is a considerable body of theoretical physics results (see, e.g.,~\cite{canetetal10}).



\subsection{Application to directed polymers}
Fix some $d\ge 1$, and let 
\begin{align}\label{polyphi}
\phi(u) :=  \log \biggl(\frac{1}{2d+1}\sum_{a\in A} e^{u_a}\biggr),
\end{align}
where $A = \{0,\pm e_1,\ldots,\pm e_d\}$, as before. It is straightforward to verify that $\phi$ is equivariant under constant shifts, zero at the origin, monotone, symmetric, twice continuously differentiable, and has a nonzero Hessian matrix at the origin. Moreover, its Hessian matrix is positive semidefinite everywhere, which shows that $\phi$ is convex. Thus, the following lemma shows that $\phi$ strictly dominates Edwards--Wilkinson growth.
\begin{lmm}\label{convlmm}
If a driving function $\phi:\rr^A \to \rr$ is equivariant under constant shifts, zero at the origin, monotone, symmetric, $C^2$ in a neighborhood of the origin, has a nonzero Hessian matrix at the origin, and is also convex, then $\phi$ satisfies the strict Edwards--Wilkinson domination condition. Moreover, it satisfies the additional condition of Theorem \ref{tightthm}.
\end{lmm}
Let $\mathbf{z} = \{z_{t,x}\}_{t\in \zz_{>0}, x\in \zz^d}$ be a collection of i.i.d.~random variables (called `noise variables' below), and for each $\ve >0$,  let $f_\ve$ be the discrete random surface generated according to~\eqref{growthmodel} with zero initial condition, using the driving function $\phi$ displayed in \eqref{polyphi}. A simple induction shows that 
\begin{align*}
f_\ve(t,x) &= \log \biggl[\frac{1}{(2d+1)^{t-1}}\sum_{P\in \cp_t}\exp\biggl(\ve \sum_{i=0}^{t-1}z_{t-i, x+p_i}\biggr)\biggr],
\end{align*}
where $\cp_t$ is the set of all lazy random walk paths of length $t$ starting at the origin --- that is, the set of all $P= (p_0,\ldots, p_{t-1})\in (\zz^d)^t$ such that $p_0=0$ and $|p_i-p_{i-1}|\le 1$ for each $i$, where $|\cdot|$ is the Euclidean norm. This is the log-partition function of the $(d+1)$-dimensional directed polymer model~\cite{comets17} on lazy random walk paths\footnote{The usual version of the model considers non-lazy random walk paths with $|p_i-p_{i-1}|=1$ for each $i$, but we change it to the lazy version to fit our framework. One can work with the usual version to arrive at a similar result, but that will require a slightly different --- and less pleasant --- definition of local KPZ growth, due to parity issues. Since the essential features of the model can be expected to remain the same for lazy paths, we  work with the lazy version.} of length $t-1$ at inverse temperature $\ve$, in the random environment $\mathbf{z}$. By Lemma~\ref{convlmm}, Theorem~\ref{kpzprob}, Theorem~\ref{tightthm}, and the above observations about $\phi$, we get the following result.
\begin{thm}\label{polythm}
Suppose that the noise variables are continuous, bounded, and have mean zero. Then $f_\ve$ has local KPZ growth under any scaling limit, in the sense of Theorem \ref{kpzprob}. Moreover, the gradient fields $\{\delta f_\ve(t,\cdot)\}_{t\in \zz_{\ge 0}}$ satisfy the conclusions of Theorem~\ref{tightthm}.
\end{thm}
Recall that for $d=1$, the only way to get the coefficients of the local KPZ equation in Theorem \ref{kpzprob} to not depend on $\ve$ is to have $\alpha(\ve)\propto \ve^4$, $\beta(\ve)\propto\ve^2$ and $\gamma(\ve)=$ constant. Translating this to polymer language, note that for a fixed $(t,x)$, $f^{(\ve)}(t,x)$ is the log-partition function for polymers of length $\alpha(\ve)^{-1}$ at inverse temperature $\ve$. Thus, $\alpha(\ve)\propto \ve^4$ means that for polymers of length $n$, the inverse temperature needs to be proportional to $n^{-1/4}$. This is the `intermediate disorder regime' considered in \cite{albertsetal14b}. It is interesting that this is the only possible way to scale so that we get constant coefficient in the local KPZ equation of Theorem \ref{kpzprob}. It is also intriguing that for $d\ge 2$, there is no way to scale so that the coefficients in the local KPZ equation do not vary with $n$ while the inverse temperature goes to zero as $n\to \infty$. 

In this context, it is worth noting (as pointed out by one of the referees) that  in the case of the continuous polymer model (Brownian motion paths in a regularized in space white noise environment, as in~\cite{mukherjeeetal16}), it is actually straightforward that under any arbitrary scaling, the rescaled log-partition function is the solution of a regularized KPZ equation which features similar scaling-dependent coefficients and a noise that converges to the white noise. One can see that thanks to the Feynman--Kac and It\^{o} formulas and the Brownian/white noise scaling properties (see, e.g., \cite[Section 2.3]{mukherjeeetal16}, where it is done for the diffusive scaling, but any other scaling also works). 

In $d= 2$, there are very few results about scaling limits of the directed polymer model. A scaling limit for the partition function (rather than the log-partition function considered here) of the $(2+1)$-dimensional model has been obtained in~\cite{caravennaetal17}, and the convergence of polymer paths to Brownian motion in the subcritical regime has been recently proved in~\cite{gabriel21}. There are a number of closely related results about `continuum polymers', which have been used to construct distribution-valued solutions of the KPZ equation. Rough calculations indicate that one might be able to obtain scaling limits of discrete directed polymers using similar arguments. For example, the results from \cite{chatterjeedunlap20, caravennaetal20, gu20} indicate that for $d=2$, a distribution-valued solution of the KPZ equation may be obtained, in the language of this paper, by taking $\alpha(\ve)\propto e^{-C\ve^{-2}}$ for sufficiently large $C$, $\beta(\ve)\propto \sqrt{\alpha(\ve)}$, and $\gamma(\ve)\propto \ve^{-1}$, as $\ve \to 0$.  One might argue, though, that these constructions  are not really solutions of the 2D KPZ equation, because it has been shown that they reduce to solutions to the stochastic heat equation with additive noise. The recent work~\cite{caravennaetal21}, which gives a non-Gaussian construction `at criticality', offers a more promising avenue to the construction of a `true' distribution-valued solution of the 2D KPZ equation. 

For $d\ge 3$, similar rough calculations indicate that the constructions in \cite{chatterjee21b1, magnenunterberger18, dunlapetal20, cometsetal19, cometsetal20, lygkoniszygouras22} correspond to taking the scaling limit of discrete directed polymers {\it keeping the inverse temperature fixed (and small)} while sending the spatial and temporal lattice spacings to zero in a certain way. This approach does not fit into our framework. 

A final remark about phase transitions for 2D polymers, proved in \cite{caravennaetal17}: One may wonder why a phase transition in the temperature parameter, as proved in \cite{caravennaetal17} for the 2D polymer model, does not manifest itself in Theorem \ref{polythm}. The possible reason is that Theorem \ref{polythm} is about the relation between local spatial and temporal derivatives of the height function, and not the height function itself. So, although the height may behave differently in different regimes, the behavior of its infinitesimal growth will exhibit no such transition, according to Theorem \ref{polythm}.


\subsection{Generalized discrete KPZ}\label{gensec}
In this subsection we consider a class of examples where $\phi$ is not necessarily convex, but the condition of strict Edwards--Wilkinson domination holds. These examples are natural discretizations of KPZ-like equations. 

Suppose that $\psi:\rr\to [0,\infty)$ is a function with the following properties: (a) it is $C^1$ everywhere, (b) it is $C^2$ in a neighborhood of the origin, (c) $|\psi'|$ is uniformly bounded, (d) $\psi(0)=0$, (e) $\psi''(0)\ne 0$, (f) $\psi(x)\ne 0$ for all $x\ne 0$, and (g) $\psi(x)$ bounded away from zero as $|x|\to \infty$. An example is:
\begin{align}\label{psiex}
\psi(x) &=
\begin{cases}
x^2 & \text{ if } |x|\le 1,\\
2|x|-1 &\text{ if } |x|>1.
\end{cases}
\end{align}
Let $c$ be a strictly positive real number, and define 
\[
\phi(u) := \bar{u}+ c \sum_{a\in A} \psi(u_a-\bar{u}). 
\]
Let $\mathbf{z} = \{z_{t,x}\}_{t\in \zz_{>0}, x\in \zz^d}$ be a collection of i.i.d.~random  variables, and let $f_\ve$ be defined by \eqref{growthmodel} with zero initial condition,  which can be rewritten as
\begin{align*}
&f_\ve(t+1,x) - f_\ve (t,x) \\
&= \bar{f}_\ve(t,x) - f_\ve(t,x) + c\sum_{a\in A} \psi(f_\ve(t,x+a) - \bar{f}_\ve(t,x)) + \ve z_{t+1,x},
\end{align*}
where $\bar{f}_\ve$ is the local average
\begin{align}\label{localavg}
\bar{f}_\ve(t,x) := \frac{1}{2d+1}\sum_{a\in A} f_\ve(t,x+a). 
\end{align}
In other words, the discrete time derivative of $f_\ve$ equals the sum of the discrete Laplacian, a noise term, and functions of discrete spatial derivatives. We may refer to this as a `generalized discrete KPZ equation'. Choosing $\psi(x)=x^2$ would make it exactly like a discrete KPZ equation, but that $\psi$ does not satisfy the bounded derivative condition required for the result stated below (which prevents $\phi$ from satisfying the monotonicity assumption). The $\psi$ displayed in equation~\eqref{psiex} is a close alternative.  The following result shows that $f_\ve$ has local KPZ growth under arbitrary scaling limits if $c$ is small enough (and the noise variables satisfy the three required conditions).
\begin{thm}\label{genthm}
Let $f_\ve$ be defined as above, with $\psi$ satisfying the listed conditions. Let $|\psi'|_\infty$ denote the supremum norm of $\psi'$. If $c \le (4d|\psi'|_{\infty})^{-1}$, and the noise variables are continuous, bounded, and have mean zero, then $f_\ve$ has local KPZ growth under arbitrary scaling limits, in the sense of Theorem \ref{kpzprob}. Moreover, if $\psi(x)\to \infty$ as $|x|\to \infty$, then the conclusions of Theorem \ref{tightthm} hold.
\end{thm}
This concludes the statements of results. The rest of the paper is organized as follows. A list of open problems is given in the next subsection. A sketch of the proof of Theorem \ref{kpzprob} is in Section \ref{sketchsec}. Section~\ref{whitesec} contains a discussion of space-time white noise. All proofs  are in Section \ref{kpzproof}.


\subsection{Open questions}
The main open question is to construct nontrivial solutions of the KPZ equation in $d\ge 2$, and then show that discrete processes such as directed polymers converge to these nontrivial solutions under appropriate scaling limits. This is a very hard problem, completely out of the reach of existing technology. Theorem \ref{kpzprob} gives hope that something like this can eventually be proved, because it shows that local KPZ behavior holds for any scaling limit --- and so, once an appropriate scaling is identified, convergence would probably hold, and the challenge would only be to prove nontriviality of the limit. 

Another class of open problems is to understand the stationary probability measures of the gradients fields, which are guaranteed to exist by Theorem \ref{tightthm}. Is the stationary measure unique? If not, what is the set of all stationary measures? What initial conditions lead to which stationary limits? What can be said about rates of convergence? 

The $C^2$ assumption on the driving function $\phi$ is restrictive, but is crucial for the notion of local KPZ universality considered in this paper. Is it possible to have a different formulation that allows driving functions that are not $C^2$? Such driving functions arise in many important models, such as last-passage percolation, ballistic deposition, etc.~(e.g., see~\cite{chatterjee21b, chatterjeesouganidis21}). For the same reason, it would be nice to be able to extend the framework of this paper to {\it asynchronous updates}, where each site is given an independent Poisson clock and the height is updated whenever the clock rings (such as in~\cite{gangulygheissari21}).

Removing the boundedness assumption on the noise variables is also a worthy goal. It is not clear how the proof technique of this paper can be extended to unbounded noise variables without introducing some constraints on how $\alpha(\ve)$, $\beta(\ve)$ and $\gamma(\ve)$ can vary with $\ve$.

Finally, in the setup of Subsection \ref{gensec}, it would be interesting to see if local KPZ behavior under arbitrary scaling limits hold when $\psi(x) = cx^2$ for some constant $c$, which is the `true' discretization of the KPZ equation. One can make $c$  vary with $\ve$ if that helps in reaching a nontrivial scaling limit. 

\section{Sketch of the proof of Theorem \ref{kpzprob}}\label{sketchsec}
First, note that by the equivariance property of $\phi$,
\begin{align*}
f_\ve(t+1,x) - f_\ve(t,x) &= \bar{f}_\ve(t,x) - f_\ve(t,x)  + \phi(q_\ve(t,x)) + \ve z_{t+1,x}, 
\end{align*}
where $\bar{f}_\ve$ is the local average of $f_\ve$ defined in equation \eqref{localavg}, and 
\[
q_\ve(t,x) := (f_\ve(t,x+a) - \bar{f}_\ve(t,x))_{a\in A}.
\] 
Now, if for some $x$ (possibly depending on $\ve$ and the noise variables), $f_\ve(t,x) \approx f_\ve(t,x+a)$ for all $a\in A$, then by Taylor expansion and using the facts that $\phi(0)=0$, $\nabla\phi(0)$ has all coordinates equal (by symmetry), and the $\hess \phi(0)$ has all diagonal elements equal and all off-diagonal elements equal (again, by symmetry), it follows that
\begin{align*}
\phi(q_\ve(t,x)) &= \phi(0) + \nabla \phi(0) \cdot q_\ve(t,x) \\
&\qquad + \frac{1}{2}q_\ve(t,x)\cdot \hess \phi(0)q_\ve(t,x) + \text{a remainder term},
\\
&= K \sum_{a\in A} (f_\ve(t,x+a) - \bar{f}_\ve(t,x))^2 + \text{a remainder term},
\end{align*}
where $K$ is a constant depending on $\phi$, and the remainder term is negligible compared to the first term on the right. Together with the preceding display, this gives
\begin{align}
\underbrace{f_\ve(t+1,x) - f_\ve(t,x)}_{\text{Time derivative}}  &= \underbrace{\bar{f}_\ve(t,x) - f_\ve(t,x)}_{\text{Laplacian term}}\notag\\
&\qquad + K \underbrace{\sum_{a\in A} (f_\ve(t,x+a) - \bar{f}_\ve(t,x))^2}_{\text{Gradient squared term}}\notag \\
&\qquad + \underbrace{\ve z_{t+1,x}}_{\text{Noise}} + \text{ a remainder term},\label{sketch}
\end{align}
which is local KPZ behavior, except that it holds only under the crucial assumption that $f_\ve(t,x+a)\approx f_\ve(t,x)$ for all $a\in A$. (We also need that the remainder term is negligible compared to the Laplacian term, the noise, and the sum of the Laplacian, noise and gradient squared terms, but let us ignore that for the time being.) This condition holds trivially at time $t=0$, since $f_\ve(0,x)=0$ for all $x$. If $\ve$ is small, it continues to hold for $t =1$, and inductively, for all $t$ up to a threshold depending on $\ve$. But to get local KPZ behavior under arbitrary scaling limits, we need to have $f_\ve(t,x+a)\approx f_\ve(t,x)$ for all $a\in A$ even if $t$ and $x$ are allowed to vary arbitrarily as $\ve \to 0$. The argument for this is outlined below. 

The first step is to show, using a random walk representation introduced in \cite{chatterjee21b}, that for any $x\in \zz^d$ and $1\le s\le t$,  
\begin{align*}
\sum_{y\in \zz^d} \biggl|\fpar{}{z_{s,y}}f_\ve (t,x)\biggr| &= \ve. 
\end{align*}
As a straightforward consequence of this identity, it follows that if $z_{1,y}$ is replaced by $0$ for each $y$, then the value of $f_\ve(t,x)$ changes by at most $B\ve$, where $B$ is a constant upper bound on the magnitude of the noise variables. (Here we use the assumption that the noise variables are bounded.) Note that this bound has no dependence on $t$ and $x$.

For each $t$ and $x$, let $g_{\ve}(t,x)$ be the value of $f_\ve(t,x)$ after replacing all $z_{1,y}$ by zero. Note that $g_\ve(1,x) = 0$ for each $x$. Thus, $g_\ve$ is just like $f_\ve$, except that instead of starting with an all zero initial condition at time $0$, we start with an all zero initial condition at time $1$. Thus, $g_\ve(t+1,x)$ has the same law as $f_\ve(t,x)$.  By the conclusion of the previous paragraph, this implies that
\begin{align*}
\ee(f_\ve(t+1,x)-f_\ve(t,x)) &= \ee(f_\ve(t+1,x)-g_\ve(t+1,x)) \le B\ve.
\end{align*}
The above deduction is the first main trick in the proof.\footnote{I thank Alex Dunlap for pointing out that this step bears similarity with  \cite[Proposition 3.1]{dunlap20} and \cite[Proposition 5.2]{dunlapetal21}, where it is used to control the squared gradient of solutions to the KPZ equation, which leads to stationary solutions of the stochastic Burgers equation.} The second trick is the following. Since the law of $f_\ve(t,x)$ is the same for all $x$, the above inequality gives
\begin{align*}
\ee(f_\ve(t+1,x)-\bar{f}_\ve(t,x)) &= \ee(f_\ve(t+1,x)-f_\ve(t,x)) \le B\ve.
\end{align*}
Since the noise variables have mean zero,
\begin{align*}
\ee(f_\ve(t+1,x)) &= \ee(\phi((f_\ve(t,x+a))_{a\in A}) + \ve z_{t+1,x})\\
&= \ee(\phi((f_\ve(t,x+a))_{a\in A})).
\end{align*}
Combining, and using equivariance of $\phi$ under constant shifts, we get 
\[
\ee(\phi(q_\ve(t,x))) \le B\ve,
\]
where recall that $q_\ve(t,x) = (f_\ve(t,x+a) - \bar{f}_\ve(t,x))_{a\in A}$. 
Note that the vector $q_\ve(t,x)$ belongs to the hyperplane $H := \{u\in \rr^A: \bar{u}=0\}$.  By Edwards--Wilkinson domination, $\phi$ is nonnegative everywhere on this hyperplane. Thus, for any $\eta >0$,
\[
\pp(\phi(q_\ve(t,x)) > \eta)\le \frac{\ee(\phi(q_\ve(t,x)))}{\eta} \le \frac{B\ve}{\eta}.
\]
By strict Edwards--Wilkinson domination, $\phi(u_n)\to 0$ implies $u_n\to 0$ on $H$. This is equivalent to saying that for any $\delta>0$, there exists $\eta(\delta)>0$ such that if $u\in H$ and $\phi(u) \le \eta(\delta)$, then $|u|\le \delta$.  Thus,
\begin{align*}
\pp(|q_\ve(t,x)| > \delta) \le \pp(\phi(q_\ve(t,x)) > \eta(\delta))\le \frac{B\ve}{\eta(\delta)}.
\end{align*}
Note that this bound has no dependence on $t$ and $x$. Thus, if $\ve \to 0$ and $t_\ve, x_\ve$ vary arbitrarily with $\ve$, we have  $|q_\ve(t_\ve,x_\ve)|\to 0$ in probability. This allows us to apply Taylor expansion and deduce \eqref{sketch}, even if $t$ and $x$ vary arbitrarily as $\ve \to 0$. Some more work is needed to establish that the remainder term is negligible compared to the other terms (this requires the assumption that $\hess\phi(0)\ne 0$) and that the noise term converges to white noise.

\section{Space-time white noise}\label{whitesec}
In this section, we recall the definition of space-time white noise and construct the field $\xi^{(\ve)}$ that converges in law to space-time white noise in Theorem \ref{kpzprob}. For the construction of space-time white noise, we follow the prescription outlined in \cite[Chapter 1]{janson97}.
 
Take any $n\ge 1$. For $\alpha = (\alpha_1,\ldots,\alpha_n)\in \zz_{\ge 0}^n$, define $|\alpha| := \alpha_1+\cdots+ \alpha_n$, 
and let $D^\alpha$ be the differential operator
\begin{align*}
D^\alpha := \frac{\partial^{|\alpha|}}{\partial x_1^{\alpha_1}\partial x_2^{\alpha_2}\cdots \partial x_n^{\alpha_n}},
\end{align*}
acting on $C^\infty(\rr^{n})$. Moreover, for any $x = (x_1,\ldots, x_{n})\in \rr^{n}$, let
\begin{align*}
x^\alpha := x_1^{\alpha_1}x_2^{\alpha_2}\cdots x_{n}^{\alpha_{n}}.
\end{align*}
For $\alpha, \beta \in \zz_{\ge 0}^{n}$ and $f\in C^\infty(\rr^{n})$, define the semi-norm
\begin{align*}
p_{\alpha, \beta}(f) := \sup_{x\in \rr^{n}} |x^\alpha D^\beta f(x)|. 
\end{align*}
A function $f\in C^\infty(\rr^{n})$ is called a Schwartz function if $p_{\alpha, \beta}(f)<\infty$ for every $\alpha,\beta\in \zz_{\ge 0}^{n}$. In other words, $f$ and all its derivatives are decaying faster than any polynomial at infinity. The space of Schwartz functions is denoted by $\ms(\rr^{n})$. The standard topology on $\ms(\rr^{n})$ is the topology generated by the countable family of semi-norms $\{p_{\alpha, \beta}: \alpha,\beta \in \zz_{\ge 0}^{n}\}$. This space is metrizable, for example by the metric
\eq{
d(f,g) := \sum_{\alpha, \beta} 2^{-|\alpha|-|\beta|} \frac{p_{\alpha,\beta}(f-g)}{1+p_{\alpha, \beta}(f-g)}. 
}
It is well known that under the above topology,  $\ms(\rr^{n})$ is Fr\'echet space. A continuous linear of functional on $\ms(\rr^{n})$ is called a tempered distribution. The space of tempered distributions on $\rr^{n}$ is denoted by $\ms'(\rr^{n})$. 



Let $\smallavg{\phi, f}$ denote the action of $\phi\in \ms'(\rr^{n})$ on $f\in \ms(\rr^{n})$. When $\phi$ is a bounded measurable function on $\rr^n$ and $f\in \ms(\rr^n)$, let $\smallavg{\phi,f}$ denote the usual $L^2$ inner product of $\phi$ and $f$. It is easy to check that this defines a continuous linear functional on $\ms(\rr^n)$. Thus, bounded measurable functions may be viewed as tempered distributions. 

There is a natural topology on $\ms'(\rr^n)$, called the `strong dual topology', defined as follows. Recall that a subset $B$ of a topological vector space is said to be bounded if for any open neighborhood $V$ of the origin, there is some $\lambda>0$ such that $B\subseteq \lambda V$. The strong dual topology on $\ms'(\rr^n)$ is generated by the family of seminorms
\[
q_B(\phi) := \sup_{f\in B} |\smallavg{\phi, f}|, \ \ B\subseteq \ms(\rr^n) \text{ bounded.}
\]
It turns out that $\ms'(\rr^n)$ is a countable union of Polish spaces under this topology. On such spaces, the usual notion of convergence of probability measures remains unchanged --- a sequence $\{\mu_n\}_{n \ge 1}$ of probability measures on the Borel $\sigma$-algebra of $\ms'(\rr^n)$ is said to converge to a probability measure $\mu$ if $\int F d\mu_n \to \int Fd\mu$ for every bounded continuous function $F:\ms'(\rr^n)\to \rr$.

An important fact about the weak convergence of $\ms'(\rr^n)$-valued random variables (called `random distributions') is that a sequence $\{\Phi_n\}_{n\ge 1}$ of random distributions converges in law to a random distribution $\Phi$ if and only if $\smallavg{\Phi_n, f}$ converges in law to $\smallavg{\Phi, f}$ for every Schwartz function $f$. This is a nontrivial result, due to \citet{fernique67, fernique68}. For a simplified proof, see~\cite{biermeetal18}. We will use this fact below.\footnote{I thank Abdelmalek Abdesselam for telling me about this result, and also about the strong dual topology on $\ms'(\rr^n)$, which I was unaware of.}


It is  a  consequence of the Minlos--Bochner theorem (see \cite{biermeetal18}) that if $\mathcal{E}$ is a continuous, symmetric, positive semidefinite bilinear form on $\ms(\rr^n)\times \ms(\rr^n)$, then there is a unique centered Gaussian measure $\nu$ on $\ms'(\rr^n)$ whose covariance kernel is $\mathcal{E}$. This means that for any $\phi \in \ms'(\rr^n)$ and $f\in \ms(\rr^n)$, $\smallavg{\phi,f}$ is a centered Gaussian random variable, and the covariance of $\smallavg{\phi,f}$ and $\smallavg{\phi, g}$ is $\mathcal{E}(f,g)$. 

In our context, we wish to define space-time white noise on $\rr_{>0}\times \rr^d$. So let us take $n = d+1$ and consider $\rr_{>0}\times \rr^d$ as a subset of $\rr^n$. Define the bilinear form
\begin{align}\label{edef}
\mathcal{E}(f,g) := \int_{\rr_{>0}\times \rr^d} f(t,x)g(t,x)dtdx
\end{align}
on $\ms(\rr^n)\times \ms(\rr^n)$. It is easy to verify that this is a continuous, symmetric, positive semidefinite bilinear form. Thus, there is a unique centered Gaussian measure $\nu$ on $\ms'(\rr^n)$ whose covariance kernel is given by $\mathcal{E}$. This is the law of space-time white noise on  $\rr_{>0}\times \rr^d$. 

Let us now construct the field $\xi^{(\ve)}$ needed for Theorem \ref{kpzprob}. With all notation as in Theorem \ref{kpzprob}, define, for $(t,x)\in \rr_{>0}\times \rr^d$ and $\ve >0$, 
\begin{align}\label{xidef}
\xi^{(\ve)}(t,x) := \sigma^{-1}\alpha(\ve)^{-1/2}\beta(\ve)^{-d/2} z_{\lceil \alpha(\ve)^{-1} t\rceil + 1, \lceil \beta(\ve)^{-1} x\rceil},
\end{align}
where $\sigma$ is the standard deviation of the noise variables. 
Since any realization of $\xi^{(\ve)}$ is a bounded measurable function, we can view $\xi^{(\ve)}$ as a random tempered distribution. The following proposition shows that it converges in law to white noise as $\ve \to 0$. 
\begin{prop}\label{whiteprop}
As $\ve\to 0$, the field $\xi^{(\ve)}$ converges in law to white noise on $\rr_{>0}\times \rr^d$, in the sense defined above.
\end{prop}
\begin{proof}
Take any $f\in \ms(\rr\times \rr^d)$. By the discussion above, we have to show that $\smallavg{\xi^{(\ve)}, f}$ converges in law to a Gaussian random variable with mean zero and variance $\mathcal{E}(f,f)$ as $\ve \to 0$, where $\mathcal{E}$ is defined as in \eqref{edef}. Fix $\ve >0$. For $m\in \zz_{>0}$ and $v\in \zz^d$, let $B_{m,v}$ denote the cuboid in $\rr_{>0}\times \rr^d$ consisting of all $(t,x)$ such that $\lceil \alpha(\ve)^{-1} t\rceil = m$ and $\lceil \beta(\ve)^{-1}x \rceil = v$. Note that these cuboids form a partition of $\rr_{>0}\times \rr^d$. Let 
\[
\bar{f}_{m,v} := \frac{1}{\vol(B_{m,v})}\int_{B_{m,v}} f(t,x)dtdx= \frac{1}{\alpha(\ve)\beta(\ve)^d}\int_{B_{m,v}} f(t,x)dtdx
\]
denote the average value of $f$ in $B_{n,v}$. Then by the decay properties of $f$, it is not hard to justify that
\begin{align*}
\smallavg{\xi^{(\ve)}, f} &= \sum_{m,v} \int_{B_{m,v}} \xi^{(\ve)}(t,x) f(t,x) dtdx\\
&=  \sigma^{-1}\alpha(\ve)^{-1/2}\beta(\ve)^{-d/2}\sum_{m,v}   z_{m+1,v}\int_{B_{m,v}}  f(t,x) dtdx\\
&= \sigma^{-1}\alpha(\ve)^{1/2}\beta(\ve)^{d/2}\sum_{m,v}   z_{m+1,v}\bar{f}_{m,v}. 
\end{align*}
Thus, $\smallavg{\xi^{(\ve)}, f}$ is a linear combination of i.i.d.~random variables. The required central limit theorem for  $\smallavg{\xi^{(\ve)}, f}$ now follows by standard methods (e.g., using characteristic functions) and the decay properties of $f$. The details are omitted.
\end{proof}

\section{Proofs}\label{kpzproof}
First, we prove Theorem \ref{kpzprob} and Theorem \ref{oepthm}. Throughout, we will assume that the conditions on $\phi$ and the noise variables stated in Section \ref{modelsec} hold. Fix a realization of $f_\ve$. Then, for any $t\in \zz_{\ge 0}$ and $x\in \zz^d$, define a random walk on $\zz^d$ as follows. The walk starts at $x$ at time $t$, and goes backwards in time, until reaching time $0$. If the walk is at location $y\in \zz^d$ at time $s\ge 1$, then at time $s-1$ it moves to $y+a$ with probability $\partial_a\phi((f_\ve(s-1,y+a))_{a\in A})$, for $a\in A$, where $\partial_a \phi$ is the derivative of $\phi$ in coordinate $a$ (which exists, by our assumption that $\phi$ is differentiable everywhere). By \cite[Lemma 3.1]{chatterjee21b}, these numbers are nonnegative and sum to $1$ when summed over $a\in A$. Therefore, this describes a legitimate random walk on $\zz^d$, moving backwards in time. (Incidentally, when $\phi$  corresponds to the polymer model, then the law of the above random walk, conditional on the noise variables, is given by the classical polymer measure. Therefore, this random walk generalizes the polymer random walk to a general class of growth models.)

The following result is a special case of  \cite[Proposition 3.2]{chatterjee21b}.

\begin{prop}\label{derivprop}
Fix a realization of $f_\ve$.  Take any $1\le s\le t$ and $x,y\in \zz^d$. Let $\{S_r\}_{0\le r\le t}$ be the backwards random walk defined above, started at $x$ at time $t$. Then 
\begin{align*}
\fpar{}{z_{s,y}}f_\ve (t,x) &= \ve \pp(S_s = y). 
\end{align*}
\end{prop}
This yields the following corollary.
\begin{cor}\label{maincor}
If $z_{1,y}$ is replaced by $0$ for each $y$, then the value of $f_\ve(t,x)$ changes by at most $\ve \max\{|z_{1,y}|: |x-y|_{1}< t\}$, where $|\cdot|_1$ denotes $\ell^1$ norm. 
\end{cor}
\begin{proof}
This is a consequence of Proposition \ref{derivprop} and the fact that if $f$ is a differentiable real-valued function on $\rr^n$ for some $n$, and $|\nabla f(x)|_{1}\le \ve$ for all $x$, then $|f(x)-f(0)|\le \ve |x|_{\infty}$, where $|\cdot|_\infty$ denotes $\ell^\infty$ norm. This holds because, by the multivariate mean-value theorem $f(x)-f(0)=x\cdot \nabla f(y)$ for some $y$ on the line joining $x$ and $0$. 
\end{proof}
The above corollary allows us to prove the following lemma.
\begin{lmm}\label{step1}
For any $t\in \zz_{\ge 0}$ and $x\in\zz^d$, 
\begin{align*}
\ee(f_\ve(t+1,x)-f_\ve(t,x)) &\le B\ve, 
\end{align*}
where $B$ is a constant upper bound on the magnitude of the noise variables. 
\end{lmm}
\begin{proof}
Let $g_{\ve}(t,x)$ be the value of $f_\ve(t,x)$ after replacing all $z_{1,y}$ by $0$. Note that $g_\ve(1,x) = 0$ for each $x$. Thus, $g_\ve$ is just like $f_\ve$, except that instead of starting with an all zero initial condition at time $0$, we start with an all zero initial condition at time $1$. This implies that $g_\ve(t+1,x)$ has the same law as $f_\ve(t,x)$, which gives
\[
\ee(f_\ve(t+1,x)-f_\ve(t,x)) = \ee(f_\ve(t+1,x) - g_\ve(t+1,x)).
\]
By Corollary \ref{maincor}, the quantity on the right is bounded by $B\ve$. 
\end{proof}

As a corollary, we obtain the following important bound.
\begin{cor}\label{gradcor}
For any $t\in \zz_{\ge 0}$ and $x\in \rr^d$, 
\begin{align*}
\ee(\phi((f_\ve(t,x+a))_{a\in A}) - \bar{f}_\ve(t,x)) \le B\ve,
\end{align*}
where $\bar{f}_\ve$ is the local average defined in equation \eqref{localavg}. 
\end{cor}
\begin{proof}
Since $f_\ve$ starts from an all zero initial condition, it follows that $\ee(f_\ve(t,y))$ does not depend on $y$. Thus,
\[
\ee(f_\ve(t,x)) = \ee(\bar{f}_\ve(t,x)).
\]
Since the noise variables have mean zero,
\begin{align*}
\ee(f_\ve(t+1,x)) &= \ee(\phi((f_\ve(t,x+a))_{a\in A}) + \ve z_{t+1,x}) \\
&= \ee(\phi((f_\ve(t,x+a))_{a\in A})). 
\end{align*}
Using the above two displays and Lemma \ref{step1}, we get the desired inequality.
\end{proof}

Our next goal is to show that $\phi(u)-\bar{u}$ grows at least quadratically in the distance of $u$ from $\bar{u}1$ when $\phi(u)-\bar{u}$ is small enough. We need two technical lemmas.
\begin{lmm}\label{nabla}
Under the assumptions on $\phi$ from Section \ref{modelsec} (specifically, symmetry, equivariance under constant shifts, and differentiability), it follows that $\nabla \phi(0) = (2d+1)^{-1}1$. 
\end{lmm}
\begin{proof}
It follows from \cite[Lemma 3.1]{chatterjee21b} that the coordinates of $\nabla \phi(0)$ sum to $1$. By symmetry, the coordinates are equal. This proves the result.
\end{proof}
\begin{lmm}\label{hess}
Let $\hess \phi(0)$ denote the Hessian matrix of $\phi$ at the origin. Then the diagonal entries of $\hess \phi(0)$ are all equal, and the off-diagonal entries are also all equal. If $q$ denotes the common value of the diagonal entries, and $r$ denotes the common value of the off-diagonal entries, then $q + 2d r = 0$. Moreover, $q$, $r$, and $q-r$ are nonzero.  
\end{lmm}
\begin{proof}
The symmetry of $\phi$ ensures the equality of all diagonal entries of $\hess \phi(0)$, and also the equality of all off-diagonal entries. Next, for $t\in \rr$, let $g(t) := \phi(t1)$. By the equivariance property, $g(t) = \phi(0)+t = t$, and hence $g''(t) \equiv 0$. On the other hand, simple calculation using solely the identity $g(t) = \phi(t1)$ shows that $g''(0) = 1\cdot \hess\phi(0)1$. Therefore, we get $1\cdot \hess\phi(0)1 = 0$, which is the same as~$q+2dr = 0$. By the nondegeneracy assumption, at least one of $q$ and $r$ is nonzero. But then, the identity $q+2dr = 0$ implies that both of them must be nonzero. Consequently, $q-r = 2dr - r = (2d-1)r$ is also nonzero. 
\end{proof}
Armed with the above lemmas, we are now ready to prove the following key fact. The proof uses the assumption of strict Edwards--Wilkinson domination.
\begin{lmm}\label{philower}
There exist $M>0$ and $c>0$ such that if $\phi(u) -\bar{u} \le M$, then
\[
\phi(u) - \bar{u} \ge c|u -\bar{u} 1|^2. 
\]
\end{lmm}
\begin{proof}
Suppose that the claim is not true. Then for any positive $M$ and $c$, there is some $u$ such that $\phi(u)-\bar{u} \le M$, but $\phi(u)-\bar{u} < c |u-\bar{u}1|^2$. For each $n$, find such a point $u_n$ for $M= c = 1/n$. Since $\phi(u_n) \ge \bar{u}_n$ (by Edwards--Wilkinson domination), this implies that $\phi(u_n) - \bar{u}_n \to 0$ and 
\[
\frac{ |u_n-\bar{u}_n1|^2}{n} >  \phi(u_n) - \bar{u}_n \ge 0. 
\]
Thus, we can divide throughout by $|u_n-\bar{u}_n1|^2$, and get
\begin{align}\label{uneq}
\frac{\phi(u_n) - \bar{u}_n}{|u_n - \bar{u}_n 1|^2} \to 0.
\end{align}
Let $y_n := u_n - \bar{u}_n1$. Since $\phi(u_n) -\bar{u}_n \to 0$, strict Edwards--Wilkinson domination gives us $y_n \to 0$. Also, note that $\bar{y}_n = 0$, $\phi(0)=0$, and by Lemma \ref{nabla}, $\nabla \phi(0) = (2d+1)^{-1} 1$. So, by the equivariance property of $\phi$ and Taylor expansion (recalling that $y_n \to 0$),  
\begin{align*}
\phi(u_n) -\bar{u}_n &= \phi(y_n)  = \phi(y_n) - \bar{y}_n\\
&= \phi(y_n) - \phi(0)-\nabla \phi(0)\cdot y_n\\
&=\frac{1}{2} y_n \cdot \hess \phi(0) y_n + o(|y_n|^2)
\end{align*}
as $n \to \infty$. Dividing both sides by $|y_n|^2$, and letting $z_n := y_n/|y_n|$, we get
\[
\frac{1}{2} z_n \cdot \hess \phi(0) z_n = \frac{\phi(u_n) - \bar{u}_n}{|u_n - \bar{u}_n1|^2} + o(1),
\]
which, by \eqref{uneq}, implies that $z_n\cdot \hess \phi(0) z_n \to 0$. But $|z_n|=1$ for each $n$, and so, passing to a subsequence if necessary, we may assume that $z_n \to z$ for some $z$ with $|z|=1$. Then $z\cdot \hess \phi(0) z = 0$. By Lemma \ref{hess}, this is the same as 
\begin{align*}
(q-r) |z|^2 + r\bar{z}^2 = 0,
\end{align*}
where $q$ and $r$ are as in Lemma \ref{hess}. But $\bar{z}_n = 0$ for each $n$, and so $\bar{z}=0$. Also, by Lemma \ref{hess}, $q-r\ne 0$. Thus, the above display shows that $z=0$, giving a contradiction to the prior observation that $|z|=1$. This completes the proof.
\end{proof}

Henceforth, let us fix two collections $\{t_\ve\}_{\ve >0}$ and $\{x_\ve\}_{\ve >0}$ in $\zz_{>0}$ and $\zz^d$, respectively. We make no assumptions about these collections; they can be completely arbitrary. Let us define some quantities whose behaviors, as $\ve \to 0$, will be of interest to us. Let
\begin{align*}
A_\ve &:= \bar{f}_\ve(t_\ve,x_\ve) - f_\ve(t_\ve, x_\ve), \\
B_\ve &:= \frac{1}{2}(q-r)\sum_{a\in A} (f_\ve(t_\ve, x_\ve+a) - \bar{f}_\ve(t_\ve,x_\ve))^2,\\
C_\ve &:= \ve z_{t_\ve+1,x_\ve},\\
D_\ve &:= f_\ve(t_\ve + 1, x_\ve) - f_\ve(t_\ve,x_\ve) - A_\ve - B_\ve - C_\ve. 
\end{align*}
We now prove a series of lemmas about these quantities. 
A general fact that we will use a number of times is the following.
\begin{lmm}\label{oplmm}
Let $\{X_\ve\}_{\ve >0}$ and $\{Y_\ve\}_{\ve >0}$ be two collections of random variables defined on the same probability space and $\{c_\ve\}_{\ve >0}$ and $\{d_\ve\}_{\ve >0}$ be two collections of positive real numbers. If $X_\ve = o_P(c_\ve)$ and $Y_\ve = O_P(d_\ve)$, then $X_\ve Y_\ve = o_P(c_\ve d_\ve)$. 
\end{lmm}
\begin{proof}
Take any $\delta, \eta >0$. Since $Y_\ve = O_P(d_\ve)$, there exists $K$ so large that 
\[
\limsup_{\ve \to 0} \pp(|Y_\ve| > Kd_\ve) \le \eta. 
\]
Then 
\begin{align*}
\pp(|X_\ve Y_\ve| > \delta c_\ve d_\ve) &\le \pp(|X_\ve| > K^{-1}\delta c_\ve) + \pp(|Y_\ve| > Kd_\ve).
\end{align*}
This shows that 
\begin{align*}
&\limsup_{\ve \to 0} \pp(|X_\ve Y_\ve| > \delta c_\ve d_\ve) \\
&\le \limsup_{\ve \to 0} \pp(|X_\ve| > K^{-1}\delta c_\ve) + \limsup_{\ve \to 0} \pp(|Y_\ve| > Kd_\ve)\le \eta. 
\end{align*}
Since $\eta$ is arbitrary, the left side must be equal to zero. Since $\delta$ is arbitrary, this proves that $X_\ve Y_\ve = o_P(c_\ve d_\ve)$. 
\end{proof}
\begin{lmm}\label{blmm}
As $\ve \to 0$, $B_\ve = O_P(\ve)$. 
\end{lmm}
\begin{proof}
By Corollary \ref{gradcor} and Edwards--Wilkinson domination, 
\begin{align}\label{phismall}
\Phi_\ve := \phi((f_\ve(t_\ve,x_\ve+a))_{a\in A}) - \bar{f}_\ve(t_\ve,x_\ve) = O_P(\ve)
\end{align}
as $\ve \to 0$.  Take any $\delta >0$. By \eqref{phismall}, there is exist $K>0$ and $\ve_0>0$ such that for any $\ve \in (0,\ve_0)$,
\begin{align}\label{phive}
\pp(\Phi_\ve > K\ve) \le \delta. 
\end{align}
Let $M$ and $c$ be as in Lemma \ref{philower}. Take any $\ve \le M/K$. Then if $\Phi_\ve \le K\ve\le M$, Lemma \ref{philower} gives 
\[
|B_\ve| \le \frac{|q-r|}{2c} \Phi_\ve \le \frac{|q-r|}{2c}K\ve. 
\]
Thus, by \eqref{phive}, for $\ve < \min\{M/K, \ve_0\}$, we have 
\begin{align*}
\pp\biggl(|B_\ve| >  \frac{|q-r|}{2c}K\ve\biggr) &\le \pp(\Phi_\ve > K\ve) \le \delta,
\end{align*}
which shows that the $\limsup$ of the left side as $\ve \to 0$ is also bounded by $\delta$. 
Thus, $B_\ve = O_P(\ve)$ as $\ve \to 0$, according to the above definition of the $O_P$ notation. 
\end{proof}
\begin{lmm}\label{deplmm}
As $\ve \to 0$, $D_\ve = o_P(B_\ve)$. 
\end{lmm}
\begin{proof}
First, note that by the equivariance property of $\phi$, 
\begin{align}
f_\ve(t_\ve + 1, x_\ve) - f_\ve(t_\ve,x_\ve)  &= \phi((f_\ve(t_\ve, x_\ve +a))_{a\in A}) + C_\ve - f_\ve(t_\ve,x_\ve)\notag\\
&=  \phi((f_\ve(t_\ve, x_\ve +a))_{a\in A}) + C_\ve + A_\ve - \bar{f}_\ve(t_\ve,x_\ve)\notag\\
&=\phi(Q_\ve) + C_\ve + A_\ve,\label{newf}
\end{align}
where
\[
Q_\ve := (f_\ve(t_\ve, x_\ve +a) - \bar{f}_\ve(t_\ve,x_\ve))_{a\in A}.
\]
Using Taylor expansion, the assumption that $\phi(0)=0$, the observation that $1\cdot Q_\ve = 0$, and the formulas for $\nabla \phi(0)$ and $\hess\phi(0)$ from Lemma \ref{nabla} and Lemma \ref{hess}, we get
\begin{align*}
|\phi(Q_\ve) - B_\ve| &= \biggl|\phi(Q_\ve) - \phi(0)-\nabla\phi(0)\cdot Q_\ve -\frac{1}{2}Q_\ve \cdot \hess\phi(0) Q_\ve\biggr|\\
&\le |Q_\ve|^2 h(|Q_\ve|),
\end{align*}
where $h:[0,\infty)\to [0,\infty)$ is a function such that $h(x)\to 0$ as $x\to0$. Since $q\ne r$ and $B_\ve = \frac{1}{2}(q-r)|Q_\ve|^2$, the above inequality and the fact that $h(x)\to 0$ as $x\to 0$ show that for any $\eta>0$ there is some $\delta>0$ such that for any $\ve$,
\begin{align*}
\pp(|\phi(Q_\ve)-B_\ve| > \eta |B_\ve|) &\le \pp(h(|Q_\ve|) > \eta|q-r|/2)\\
&\le \pp(|Q_\ve| > \delta).
\end{align*} 
But by Lemma \ref{blmm}, $|Q_\ve|\to 0$ in probability as $\ve \to 0$. Thus, the last expression in the above display tends to zero as $\ve \to 0$. This shows that $\phi(Q_\ve) = B_\ve + o_P(B_\ve)$. By \eqref{newf}, this completes the proof of the lemma. 
\end{proof}
As a corollary of the three lemmas above, we immediately get the following. This will be useful later. 
\begin{cor}\label{dep}
As $\ve \to 0$, $D_\ve = o_P(\ve)$. 
\end{cor}
Next, we prove `lower bounds in probability' for $A_\ve$, $C_\ve$, and $A_\ve + B_\ve + C_\ve$. In the following, $z$ denotes a random variable following the law of the noise variables.
\begin{lmm}\label{cep}
As $\ve \to 0$, $C_\ve^{-1} = O_P(\ve^{-1})$.
\end{lmm}
\begin{proof}
Take any $K>0$. Then
\begin{align*}
\pp(|\ve C_\ve^{-1}|>K) &= \pp(|z_{t_\ve +1, x_\ve}|^{-1} > K)\\
&= \pp(|z| < K^{-1}),
\end{align*}
which tends to zero as $K\to \infty$, since the law of $z$ is absolutely continuous with respect to Lebesgue measure.
\end{proof}
\begin{lmm}\label{aep}
As $\ve \to 0$, $A_\ve^{-1} = O_P(\ve^{-1})$. 
\end{lmm}
\begin{proof}
Note that $\ve^{-1} A_\ve$ can be written as $bz_{t_\ve, x_\ve} + R_\ve$, where $R_\ve$ and $z_{t_\ve, x_\ve}$ are  independent, and $b = -2d/(2d+1)$. Since the law of $z$ is absolutely continuous with respect to Lebesgue measure, it is a standard fact that for any $\delta>0$ there is some $\eta >0$ such that $\pp(z\in S)< \delta$ for any Borel set $S$ with Lebesgue measure  less than $\eta$. This shows that for any $K>0$ and $r\in \rr$,
\begin{align*}
\pp(|\ve A_\ve^{-1}| > K \mid R_\ve = r) &= \pp(|bz_{t_\ve, x_\ve} + R_\ve| < K^{-1} \mid R_\ve = r)\\
&= \pp(|bz_{t_\ve, x_\ve} + r| < K^{-1} \mid R_\ve = r)\\
&= \pp(|bz + r| < K^{-1})\le f(K),
\end{align*}
where $f(K)$ is a function only of $K$, with no dependence on $r$, that tends to zero as $K\to\infty$. Since $f(K)$ has no dependence on $r$ and $\ve$, we can take expectation over $r$ on the left side and arrive at the desired result.
\end{proof}
\begin{lmm}\label{abcep}
As $\ve \to 0$, $(A_\ve + B_\ve+C_\ve)^{-1} = O_P(\ve^{-1})$. 
\end{lmm}
\begin{proof}
Note that $\ve^{-1}(A_\ve+B_\ve+C_\ve)$ can be written as $z_{t_\ve +1,x_\ve} + Q_\ve$, where $Q_\ve$ and $z_{t_\ve +1,x_\ve}$ are independent. The rest of the proof proceeds exactly as in the proof of Lemma \ref{aep}.
\end{proof}
Combining Corollary \ref{dep}, and Lemmas \ref{oplmm}, \ref{cep}, \ref{aep}, and~\ref{abcep}, we obtain the following result.
\begin{cor}\label{depcor}
As $\ve \to 0$, $D_\ve$ is $o_P$ of $A_\ve$, $C_\ve$, and $A_\ve+B_\ve+C_\ve$. 
\end{cor}
We now have all the ingredients for the proofs of Theorem \ref{kpzprob} and Theorem \ref{oepthm}.
\begin{proof}[Proof of Theorem \ref{kpzprob}]
Fixing $(t,x)\in \rr_{>0}\times \rr^d$, define $t_\ve := \lceil \alpha(\ve)^{-1} t\rceil$ and $x_\ve := \lceil \beta(\ve)^{-1} x\rceil$. Note that since $t>0$ and $\alpha(\ve)>0$, we have that $t_\ve \ge 1$ for any~$\ve$ (this is why we use ceiling instead of floor). Now, observe that 
\begin{align*}
\lceil\alpha(\ve)^{-1}(t+\alpha(\ve))\rceil = \lceil \alpha(\ve)^{-1} t + 1\rceil = \lceil \alpha(\ve)^{-1}t \rceil + 1 = t_\ve + 1.
\end{align*}
Similarly, for any $a\in A$,
\begin{align*}
\lceil \beta(\ve)^{-1}(x + \beta(\ve) a)\rceil = x_\ve + a. 
\end{align*}
This implies that
\begin{align*}
\tilde{\partial}_t f^{(\ve)}(t,x)  &= \frac{\gamma(\ve) }{\alpha(\ve)} (f_\ve(t_\ve+1, x_\ve) - f(t_\ve,x_\ve)).
\end{align*}
Similarly, note that
\begin{align*}
\tilde{\Delta} f^{(\ve)}(t,x) &= \frac{(2d+1)\gamma(\ve)}{\beta(\ve)^2}(\bar{f}_\ve(t_\ve, x_\ve) - f_\ve(t_\ve, x_\ve)) \\
&= \frac{(2d+1)\gamma(\ve)}{\beta(\ve)^2}A_\ve,
\end{align*}
and 
\begin{align*}
|\tilde{\nabla} f^{(\ve)}(t,x)|^2 &= \frac{\gamma(\ve)^2}{2\beta(\ve)^2} \sum_{a\in A} (f_\ve(t_\ve,x_\ve + a) - \bar{f}_\ve(t_\ve, x_\ve))^2\\
&= \frac{\gamma(\ve)^2}{(q-r)\beta(\ve)^2} B_\ve. 
\end{align*}
Let $\xi^{(\ve)}$ be defined as in equation \eqref{xidef}. Then note that
\begin{align*}
\xi^{(\ve)}(t,x) &= \sigma^{-1}\alpha(\ve)^{-1/2}\beta(\ve)^{-d/2} \ve^{-1} C_\ve. 
\end{align*}
Finally, let
\[
R^{(\ve)}(t,x) := \frac{\gamma(\ve)}{\alpha(\ve)}D_\ve. 
\]
Using all of the above, and the definition of $D_\ve$, we get
\begin{align*}
\tilde{\partial}_t f^{(\ve)}(t,x) &= \frac{\gamma(\ve) }{\alpha(\ve)}(A_\ve+B_\ve+C_\ve+D_\ve)\\
&= \frac{\beta(\ve)^2}{(2d+1)\alpha(\ve)}\tilde{\Delta} f^{(\ve)}(t,x)  + \frac{(q-r)\beta(\ve)^2}{\alpha(\ve)\gamma(\ve)}|\tilde{\nabla} f^{(\ve)}(t,x)|^2 \\
&\qquad + \frac{\sigma \ve \beta(\ve)^{d/2}\gamma(\ve)}{\alpha(\ve)^{1/2}}\xi^{(\ve)}(t,x) + R^{(\ve)}(t,x). 
\end{align*}
By Lemma \ref{deplmm} and Corollary~\ref{depcor}, $D_\ve$ is $o_P$ of $A_\ve$, $B_\ve$, $C_\ve$, and $A_\ve + B_\ve + C_\ve$. By Proposition \ref{whiteprop}, $\xi^{(\ve)}$ converges in law to white noise as $\ve \to 0$. This completes the proof. 
\end{proof}
\begin{proof}[Proof of Theorem \ref{oepthm}]
Note that for any $a\in A$, 
\begin{align*}
&\sum_{b\in A} (f_\ve(t_\ve, x_\ve+b) - \bar{f}_\ve(t_\ve,x_\ve))^2 \\
&= \frac{1}{4d+2}\sum_{b,c\in A} (f_\ve(t_\ve, x_\ve+b) - f_\ve(t_\ve, x_\ve +c))^2\\
&\ge \frac{1}{4d+2}(f_\ve(t_\ve, x_\ve+a) - f_\ve(t_\ve, x_\ve))^2,
\end{align*}
and apply Lemma \ref{blmm} and the fact that $q\ne r$. 
\end{proof}

Next, let us prove Theorem \ref{tightthm}. The proof requires the following lemmas.
\begin{lmm}\label{tightlmm0}
A sequence of $(\rr^d)^{\zz^d}$-valued random variables $\{f_n\}_{n\ge 1}$  is tight if and only if $\{f_n(x)\}_{n\ge 1}$ is a tight family of $\rr^d$-valued random variables for every $x\in \zz^d$.
\end{lmm}
\begin{proof}
If $\{f_n\}_{n\ge 1}$ is a tight family, then the continuity of the projection $f\mapsto f(x)$ shows that for any $x$, $\{f_n(x)\}_{n\ge 1}$ is a tight family. Conversely, suppose that $\{f_n(x)\}_{n\ge 1}$ is a tight family for each $x$. Fix some $\delta>0$. Then for every $x$, there is a compact set $K_x\subseteq \rr^d$ such that $\pp(f_n(x)\notin K_x)\le 2^{-|x|}\delta$ for all $n$. Let $K := \prod_{x\in \zz^d} K_x$. Then $K$ is a compact set under the product topology, and for any $n$, 
\[
\pp(f_n \notin K) \le \sum_{x\in \zz^d} \pp(f_n(x)\notin K_x) \le C\delta,
\]
where $C$ does not depend on $n$. This completes the proof.
\end{proof}

\begin{lmm}\label{tightlmm}
Under the hypotheses of Theorem \ref{tightthm}, the sequence $\{\delta f_\ve(t,\cdot)\}_{t\in \zz_{\ge0}}$ is a tight family.
\end{lmm}
\begin{proof}
By Lemma \ref{tightlmm0}, it suffices to prove that for each $x$, $\{\delta f_\ve(t,x)\}_{t\in \zz_{\ge 0}}$ is a tight family of random vectors. For this, it is necessary and sufficient to have that $\{\delta_i f_\ve(t,x)\}_{t\in \zz_{\ge 0}}$ is a tight family of real-valued random variables. Fix some $x$ and~$i$. By Corollary \ref{gradcor}, $\ee(\phi(q_\ve(t,x)))\le B\ve$, where
\begin{align*}
q_\ve(t,x) := (f_\ve(t, x+a)- \bar{f}_\ve(t,x))_{a\in A}. 
\end{align*}
Note that $q_\ve(t,x)\in H$, where $H:= \{u\in \rr^A: \bar{u}=0\}$. By Edwards--Wilkinson domination, $\phi(u)\ge 0$ for all $u\in H$. Moreover, by the additional condition of Theorem \ref{tightthm}, we have that for any $K>0$, there is some $L>0$ such that if $u\in H$ and $|u|>L$, then $\phi(u)>K$. Thus,
\begin{align*} 
\pp(|q_\ve(t,x)| >L) \le \pp(\phi(q_\ve(t,x)) > K) &\le \frac{\ee(\phi(q_\ve(t,x)))}{K} \le \frac{B\ve}{K}. 
\end{align*}
By the inequality displayed in the proof of Theorem \ref{oepthm}, this proves the tightness of $\{\delta_i f_\ve(t,x)\}_{t\in \zz_{\ge0}}$. 
\end{proof}

\begin{lmm}\label{markovlmm}
The sequence $\{\delta f_\ve(t,\cdot)\}_{t\in \zz_{\ge0}}$ is a time-homogeneous Markov chain.
\end{lmm}
\begin{proof}
Note that by the equivariance property of $\phi$,
\begin{align*}
&f_\ve(t+1,x+e_i) - f_\ve(t+1,x) \\
&= \phi((f_\ve(t,x+e_i +a))_{a\in A}) + \ve z_{t+1,x+e_i} \\
&\qquad \qquad - \phi((f_\ve(t,x+a))_{a\in A}) - \ve z_{t+1,x}\\
&= \phi((f_\ve(t,x+e_i +a) - f_\ve(t, x+e_i))_{a\in A}) - \phi((f_\ve(t,x+a) - f_\ve(t,x))_{a\in A}) \\
&\qquad f_\ve(t,x+e_i) - f_\ve(t,x) + \ve z_{t+1,x+e_i}  - \ve z_{t+1,x}.
\end{align*}
This shows that $\delta f_\ve(t+1,\cdot)$ is a function of $\delta f_\ve(t,\cdot)$ and $\{z_{t+1,x}\}_{x\in \zz^d}$, from which it is clear that $\{\delta f_\ve(t,\cdot)\}_{t\in \zz_{\ge0}}$ is a time-homogeneous Markov chain. 
\end{proof}
Let $T$ denote the transition kernel of the Markov chain from Lemma~\ref{markovlmm}. That is, for a probability measure $\mu$ on $(\rr^d)^{\zz^d}$, $T\mu$ denotes the probability law after taking one step from the chain if the initial state has law $\mu$. 
\begin{lmm}\label{tcontlmm}
The map $T$ defined above is continuous on the space of probability measures on $(\rr^d)^{\zz^d}$ under the topology of weak convergence.
\end{lmm}
\begin{proof}
Let $\Psi$ be a bounded continuous function from $(\rr^d)^{\zz^d}$ into $\rr$. Let $\{\mu_n\}_{n\ge 1}$ be a sequence of probability measures on $(\rr^d)^{\zz^d}$ converging weakly to a probability measure $\mu$. Let $\nu_n := T\mu_n$ and $\nu := T\mu$. For each $n$, let $f_n$ be a $(\rr^d)^{\zz^d}$-valued random variable with law $\mu_n$. Let $z := \{z_x\}_{x\in \zz^d}$ be a collection of i.i.d.~random variables having the same law as our noise variables, independent of the $f_n$'s. Then, since $\phi$ is differentiable everywhere --- and hence, continuous --- it is not hard to see that there is a continuous function $\Phi: (\rr^d)^{\zz^d}\times \rr^d\to (\rr^d)^{\zz^d}$ such that $\Phi(f_n, z)$ has law $\nu_n$. 

Now, note that $(f_n,z)$ converges in law to $(f,z)$, where $f$ has law $\mu$ and is independent of $z$. Since $\Psi\circ \Phi$ is a bounded continuous function, this implies that
\begin{align*}
\int \Psi d\nu_n &= \ee(\Psi(\Phi(f_n, z))) \to \ee(\Psi(\Phi(f,z)) = \int \Psi d\nu,
\end{align*}
Thus, $\nu_n\to \nu$ weakly, which completes the proof.
\end{proof}  
We are now ready to prove Theorem \ref{tightthm}.
\begin{proof}[Proof of Theorem \ref{tightthm}]
Let $\gamma_t$ be the law of $\delta f_\ve(t,\cdot)$. Define
\[
\mu_t := \frac{1}{t}\sum_{s=0}^{t-1}\gamma_s.
\]
By Lemma \ref{tightlmm}, $\{\gamma_t\}_{t\in \zz_{\ge 0}}$ is a tight family. From this, it follows that $\{\mu_t\}_{t\in \zz_{\ge0}}$ is also a tight family. Therefore, by Prokhorov's theorem, it has a weakly convergent subsequence. Passing to this subsequence if necessary, let us assume that $\mu_t$ converges weakly to some $\mu$. We claim that $\mu$ is an invariant probability measure for the Markov kernel $T$. To see this, let $\nu_t := T \mu_t$ and $\lambda_t := T \gamma_t$. Let $\Psi :(\rr^d)^{\zz^d}\to \rr$ be a bounded continuous function. Then by the linearity of $T$, 
\begin{align*}
\int \Psi d\nu_t &= \frac{1}{t}\sum_{s=0}^{t-1} \int \Psi d\lambda_s. 
\end{align*}
But for each $t$, $\lambda_t = \gamma_{t+1}$. Thus, 
\begin{align*}
\int \Psi d\nu_t &= \frac{1}{t}\sum_{s=0}^{t-1} \int \Psi d\gamma_{s+1}\\
&= \frac{1}{t}\sum_{s=0}^{t-1} \int \Psi d\gamma_{s} + \frac{1}{t}\biggl(\int \Psi d\gamma_t - \int\Psi d\gamma_0\biggr)\\
&= \int\Psi d\mu_t  + \frac{1}{t}\biggl(\int \Psi d\gamma_t - \int\Psi d\gamma_0\biggr).
\end{align*}
By the boundedness of $\Psi$, the second term on  the right goes to zero as $t\to \infty$. By assumption, $\mu_t \to \mu$, and so by Lemma \ref{tcontlmm}, $\nu_t \to \nu := T\mu$. Combining, we get that $\int \Psi d\nu = \int\Psi d\mu$. Since $\Psi$ is an arbitrary bounded continuous function, this shows that $\nu=\mu$. Thus, $\mu$ is an invariant probability measure for the kernel~$T$. The translation invariance of $\mu$ follows from the translation invariance of each $\gamma_t$. 
\end{proof}

Next, let us prove Lemma \ref{convlmm}. We need the following lemma.
\begin{lmm}\label{convsupport}
Suppose that the hypotheses of Lemma \ref{convlmm} hold. Then there exists $\delta>0$, depending only $\phi$, such that for any $u\in\rr^A$ with $\bar{u}=0$, we have
\[
\phi(u) \ge \frac{1}{4}(q-r)|u| \min\{\delta, |u|\},
\]
where $q$ and $r$ are as in Lemma \ref{hess}. Moreover, we have that $q>r$.
\end{lmm}
\begin{proof}
By Lemma \ref{hess} (whose proof does not use Edwards--Wilkinson domination), we have that for any $u$ with $\bar{u}=0$, 
\begin{align}\label{hessid}
u\cdot \hess\phi(0) u = (q-r) |u|^2.
\end{align}  
Since $\hess\phi(0)$ is positive semidefinite due to the convexity of $\phi$, this immediately shows that $q\ge r$. By Lemma \ref{hess}, $q\ne r$. Thus, $q>r$.  

In the following, $|M|$ denotes the Euclidean norm of a matrix $M$ --- that is, the square-root of the sum of squares of the entries. Since $\phi$ is $C^2$ in a neighborhood of the origin and $q>r$, there exists $\delta$ small enough such that $\phi$ is $C^2$ in the open ball of radius $2\delta$ centered at the origin, and $|\hess \phi(u)- \hess \phi(0)|< (q-r)/2$ for all $u$ in this ball. 

Take any $u\in \rr^A$ such that $\bar{u} = 0$ and $|u|\le \delta$. For $t\in [0,1]$, let $g(t) := \phi(tu)$. Then 
\begin{align*}
\phi(u) &= g(1) = g(0) + g'(0) + \int_0^1 (1-t) g''(t) dt.
\end{align*}
Now, $g(0)=\phi(0)=0$, and by Lemma \ref{nabla}, $g'(0) = \nabla \phi(0)\cdot u = \bar{u}=0$. By definition of $g$, $g''(t) = u\cdot \hess \phi(tu) u$. Inserting these into the above expression, we get
\begin{align}\label{phiform}
\phi(u) &= \int_0^1 (1-t) (u\cdot \hess \phi(tu) u) dt. 
\end{align}
Now, for all $t\in [0,1]$, an application of the Cauchy--Schwarz inequality gives
\begin{align*}
|u\cdot \hess\phi(tu) u - u\cdot \hess\phi(0) u| &\le |\hess \phi(tu) - \hess\phi(0)| |u|^2 \\
&\le \frac{1}{2}(q-r)|u|^2. 
\end{align*}
By \eqref{hessid} and the above inequality, we have that for all $t\in [0,1]$,
\begin{align}\label{hessineq}
u\cdot \hess\phi(tu) u \ge \frac{1}{2}(q-r)|u|^2.
\end{align}
By \eqref{phiform}, this gives
\begin{align}\label{phiu1}
\phi(u) \ge \frac{1}{4}(q-r)|u|^2. 
\end{align}
Next, suppose that $|u| >\delta$. Let $v := \alpha u$, where $\alpha := \delta/|u|$. Then $\bar{v}=0$ and $|v| = \delta$. Thus, by \eqref{phiu1},
\begin{align}\label{phiv}
\phi(v) \ge \frac{1}{4}(q-r)|v|^2 = \frac{1}{4} (q-r) \delta^2. 
\end{align}
But, by the convexity of $\phi$,
\[ 
\phi(v) = \phi(\alpha u) \le \alpha \phi(u) + (1-\alpha) \phi(0) = \alpha \phi(u).
\]
Thus, by \eqref{phiv}, 
\begin{align}\label{phiu2}
\phi(u) \ge \alpha^{-1}\phi(v) \ge \frac{1}{4}(q-r)\delta|u|. 
\end{align}
Combining \eqref{phiu1} and \eqref{phiu2} completes the proof of the lemma.
\end{proof}
We are now ready to complete the proof of Lemma \ref{convlmm}. 
\begin{proof}[Proof of Lemma \ref{convlmm}]
By convexity, $\phi(u) - \phi(0) - \nabla \phi(0) \cdot u \ge 0$ for all $u$. But $\phi(0)=0$, and by Lemma \ref{nabla} (which uses only equivariance, monotonicity, and symmetry in its proof), $\nabla \phi(0) = (2d+1)^{-1}1$. Thus, $\phi(u)-\bar{u}\ge 0$ for all $u$. This proves Edwards--Wilkinson domination. To prove strict domination, let $\{u_n\}_{n\ge 1}$ be a sequence such that $\phi(u_n) -\bar{u}_n \to 0$. Let $v_n := u_n - \bar{u}_n1$, so that $\bar{v}_n =0$ for all $n$ and $\phi(v_n)\to 0$. Then by Lemma~\ref{convsupport}, we have that $v_n \to 0$, which means  that $u_n - \bar{u}_n 1 \to 0$. Finally, by Lemma \ref{convsupport}, we see immediately that the extra condition of Theorem \ref{tightthm} is satisfied.  
\end{proof}

Finally, let us prove Theorem \ref{genthm}.

\begin{proof}[Proof of Theorem \ref{genthm}]
The claims follow from Theorem \ref{kpzprob} and Theorem \ref{tightthm}, if we can just verify that $\phi$ satisfies the necessary conditions. 
It is easy to see that $\phi$ is equivariant under constant shifts, symmetric, and zero at the origin. A simple calculation shows that any mixed partial derivative of $\phi$ at the origin is equal to $C(d)c\psi''(0)$, where $C(d)$ is a nonzero constant depending only on $d$. Since $c>0$ and $\psi''(0)\ne 0$, this shows that $\phi$ satisfies the nondegeneracy condition. Next, note that 
\[
\fpar{\phi}{u_a} = \frac{1}{2d+1} + c \psi'(u_a - \bar{u}) - \frac{c}{2d+1} \sum_{b\in A} \psi'(u_b - \bar{u}).
\]
By the uniform boundedness of $|\psi'|$, the above expression shows that $\phi$ is monotone if we choose $c$ small enough --- specifically, if $c\le (4d|\psi'|_\infty)^{-1}$. Next, let us show that $\phi$ satisfies the strict Edwards--Wilkinson domination condition. Since $c>0$ and $\psi\ge0$ everywhere, we have that $\phi(u)\ge \bar{u}$ for all $u$. Next, take any sequence $\{u_n\}_{n\ge 1}$ such that $\phi(u_n)-\bar{u}_n \to 0$. Suppose that $\{u_n - \bar{u}_n1\}_{n\ge 1}$ is an unbounded sequence. Since $\psi(x)$ is bounded away from zero as $|x|\to \infty$ and $\psi\ge 0$ everywhere, this implies that $\phi(u_n)-\bar{u}_n$ cannot converge to zero, contradicting our hypothesis. Thus, $\{u_n - \bar{u}_n1\}_{n\ge 1}$ must be a bounded sequence. Since $\psi$ is continuous and nonnegative, and the only point where it is zero is the origin, we conclude that any convergent subsequence of $\{u_n - \bar{u}_n1\}_{n\ge 1}$ must converge to zero. Thus, $u_n - \bar{u}_n1\to 0$. This proves that $\phi$ satisfies the strict Edwards--Wilkinson domination condition. Finally, if $\psi(x)\to \infty$ as $|x|\to \infty$, then it is clear, by the nonnegativity of $\psi$, the function $\phi$ satisfies the extra condition of Theorem~\ref{tightthm}.
\end{proof}

\section*{Acknowledgements}
I thank Abdelmalek Abdesselam, Alex Dunlap, Herbert Spohn, Kevin Yang, and the anonymous referees for a number of useful comments and references. 

\bibliographystyle{plainnat}
\bibliography{myrefs}

\end{document}